\documentclass{amsart}
\usepackage{amsmath,amssymb,amsthm}
\usepackage{tikz}
\usepackage{hyperref}
\newtheorem{theorem}{Theorem}[section]
\newtheorem{prop}[theorem]{Proposition}
\newtheorem{corollary}[theorem]{Corollary}
\newtheorem{lemma}[theorem]{Lemma}
\theoremstyle{definition}
\newtheorem{defi}[theorem]{Definition}
\theoremstyle{remark}
\newtheorem{Rk}{Remark}[section]

\DeclareMathOperator{\CS}{\Pi}

\newcommand{\R}{\mathbb{R}} 
\newcommand{\C}{\mathbb{C}} 
\newcommand{\N}{\mathbb{N}} 
\newcommand{\Z}{\mathbb{Z}}

\newcommand{\D}{\mathcal D}

\newcommand{\p}{\partial}
\newcommand{\eps}{\varepsilon}
\newcommand{\la}{\lambda}
\newcommand{\pv}{\textrm{p.v.}}
\DeclareMathOperator{\supp}{supp}
\DeclareMathOperator{\Id}{Id}
\DeclareMathOperator{\Range}{Range}
\DeclareMathOperator{\Ker}{Ker}
\DeclareMathOperator{\dist}{dist}
\DeclareMathOperator{\tr}{tr}

\newcommand{\qtq}[1]{\quad\text{#1}\quad}

\let\Re=\undefined\DeclareMathOperator{\Re}{Re}
\let\Im=\undefined\DeclareMathOperator{\Im}{Im}

\numberwithin{equation}{section}

\begin{document}

\title{Asymptotic stability of Benjamin--Ono multisolitons in $L^2(\R)$}

\author{Rana Badreddine}
\address{Department of Mathematics, University of California, Los Angeles, CA 90095, USA.}
\email{badreddine@math.ucla.edu}

\author{Rowan Killip}
\address{CEREMADE, CNRS, Universit\'e Paris Dauphine–PSL, Place du Mar\'echal de Lattre de Tassigny, 75016 Paris, France \&  Department of Mathematics, University of California, Los Angeles, CA 90095, USA}
\email{killip@ceremade.dauphine.fr}

\author{Monica Vi\c{s}an}
\address{Department of Mathematics, University of California, Los Angeles, CA 90095, USA.}
\email{visan@math.ucla.edu}

\begin{abstract}
We prove the following dichotomy result for $L^2(\R)$ solutions to the Benjamin--Ono equation: On windows traveling at any speed, the solution either converges to zero or to a soliton dictated by the spectral properties of the Lax operator associated to the initial data. As an application of this result, we prove asymptotic stability of Benjamin--Ono multisolitons in $L^2(\R)$. Specifically, we show that solutions to the Benjamin--Ono equation emanating from small $L^2(\R)$ perturbations of multisolitons evolve towards a series of separating one-solitons when viewed in windows traveling with these solitons.
\end{abstract}

\maketitle

\tableofcontents

\maketitle

\section{Introduction}\label{S:1}

This paper investigates the asymptotic behavior of solutions to the Benjamin--Ono equation,
\begin{equation}\label{BO}\tag{BO}
    \partial_t u = H\partial_x^2 u - 2u\partial_x u, \qquad u: \R_t\times\R_x\to\R.
\end{equation}
Here $H$ denotes the Hilbert transform; see \eqref{H def}.  This model was introduced in the late 1960s by Benjamin \cite{Benjamin1967} and Davis--Acrivos \cite{Davis1967} to describe interfacial waves in stratified fluids of large total depth. The real-valued function $u(t,x)$ represents the deviation of the interface from equilibrium.

One of the important predictions of \eqref{BO}, observed already in these initial studies, was that it supports solitary waves. Concretely,
\begin{equation}\label{1-soliton}
Q_{\lambda,c}(t,x) = \frac{-4\lambda}{1+ 4\lambda^2(x-c+2\lambda t)^2}
\end{equation}
is a solution for every $\lambda<0$ and $c\in\R$.  Indeed, both \cite{Benjamin1967,Davis1967} confirmed the existence of such solitary waves experimentally.  The ease with which such waves were generated and their apparent stability prompted Ono \cite{Ono1975} to posit that these are solitons.  A key characteristic of solitons is that they exhibit particle-like behavior under interactions.  This was soon confirmed by the construction of exact multisoliton solutions; see \cite{Case1979,MR516327,Joseph1977,Matsuno1979,MeissPereira}.  These nonlinear structures will be important characters in our story; the following description of them is adapted from \cite[\S3.1]{MR759718}:

\begin{defi}[$N$-Multisolitons]\label{D:multisoliton}
Given an integer $N\geq 1$, a set of negative spectral parameters $\Lambda= \{\la_1< \cdots <\la_N\}$, and spatial centers $\vec c =(c_1, \ldots, c_N)\in \R^N$, we define a corresponding \emph{$N$-soliton profile} by 
\begin{align}\label{MS1}
Q_{\Lambda, \vec c} \,(x)= - 2\Im \frac{d}{dx} \ln \det\bigl[M(x)\bigr] , 
\end{align}
where $M(x)$ is the $N\times N$ matrix with entries
\begin{align}\label{MS2}
M_{jk}(x) = \begin{cases}
-i(x-c_j) - \frac{1}{2\la_j}, \qquad &\text{if } j=k,\\
-\frac{1}{\la_j-\la_k},  \qquad &\text{if } j\neq k.
\end{cases}
\end{align}
The \emph{$N$-soliton solution} to \eqref{BO} with initial data $Q_{\Lambda, \vec c}$ is then 
\begin{align}\label{MS3}
u(t,x) = Q_{\Lambda, \vec c(t)}(x) \qtq{with} c_j(t) = c_j - 2\la_j t.
\end{align}
\end{defi}

Our names for the two families of parameters $\Lambda$ and $\vec c$ are indicative of their physical significance, as we will now explain.  Following the KdV model, we expect solitons to be associated to discrete eigenvalues of the Lax operator.  This is indeed the case for \eqref{BO}: the parameters $\Lambda$ are the eigenvalues of the Lax operator $L_u$ defined in \eqref{Lax op} when $u=Q_{\Lambda, \vec c}$.  These parameters also dictate the speed, height, and width of the individual solitons; compare \eqref{1-soliton}.

While the solitons are overlapping, it is difficult to ascribe an exact center to each; however, they separate as $t\to\pm\infty$ and travel with speeds $2|\lambda_j|$.  Unlike KdV, the solitons do not advance or retard one another as they interact; the points $c_j(t)$ in \eqref{MS3} describe their centers in both limits $t\to\pm\infty$.

While exact multisoliton solutions indicate that individual solitons interact elastically, this is not sufficient to confirm their stability.  Exact multisolitons correspond to very specific initial data.  To prove stability of solitons, we must allow for very general perturbations.

Stability is inherently a dynamical question, so any discussion of stability must begin with the construction of solutions.  The question of well-posedness for the richest possible class of initial data has received a great deal of attention over the years, with a particular emphasis on the $H^s(\R)$ class of Sobolev spaces.  An outline of these developments can be found in \cite{Killip2024}, the main result of which is as follows:

\begin{theorem}[Global well-posedness, \cite{Killip2024}]\label{T:KLV main}
The equation~\eqref{BO} is globally well-posed in $H^s(\R)$ for all $s>-\frac12$.
\end{theorem}

By contrast, it is known that the data-to-solution map for \eqref{BO} does not admit a continuous extension to $H^s(\R)$ for any $s<-\frac12$.

Just as the well-posedness question for \eqref{BO} has received a great deal of attention, so has the stability of its soliton and multisoliton solutions.  The first type of stability to be demonstrated was \emph{orbital} stability.  This means that for initial data that is close to a (multi)soliton, the solution remains close to the manifold of \text{(multi)}solitons with the same spectral parameters $\Lambda$, but arbitrary centers $\vec c$.  This problem is often attacked via the Dirichlet--Lagrange principle: one shows that \text{(multi)}solitons are optimizers of a constrained variational problem based on conserved quantities.  Such an approach was applied to \eqref{BO} solitons in \cite{Albert1992,Albert1999,MR887857,MR715035,MR2082818,MR897729,MR886343} and to its multisolitons in \cite{BKV25,LanWang2025,Matsuno2006,MR1442235,MR2202311}.  The traditional variational approach to multisolitons, pioneered in \cite{MR1220540}, is based on the polynomial conservation laws and so requires ever-higher regularity hypotheses as the complexity $N$ of the multisoliton increases.  By comparison, the paper \cite{BKV25} proved orbital stability of all multisolitons in $H^s(\R)$ all the way down to $s>-\frac12$, which coincides with the optimal well-posedness threshold:

\begin{theorem}[Uniform orbital stability of multisolitons, \cite{BKV25}]\label{Th: Orbital Stability}
Fix $-\frac12<s\leq \frac12$, an integer $N\geq 1$, and a set of negative spectral parameters $\Lambda=\{\la_1<\cdots<\la_N\}$. For every \(\eps>0\), there exists $\delta>0$ so that for every initial data $u_0\in H^s(\R)$ satisfying
\[
      \inf_{\vec{c}\in\R^N}\,\|u_0-Q_{\Lambda,\vec{c}}\,\|_{H^s}<\delta,
\]
the corresponding solution \( u(t) \) of \eqref{BO} satisfies
\[
     \sup_{t\in\R}\inf_{\vec{c}\in\R^N}\, \|u(t)-Q_{\Lambda,\vec{c}}\,\|_{H^s}<\eps.
\]
\end{theorem}

A very different approach to orbital stability is employed in the papers \cite{GTT2009,KenMart2009}.  As noted, the explicit multisoliton solutions ultimately resolve into well-separated individual solitons ordered by their speed.  These papers demonstrate that once the solution has arrived at this favorable configuration, it will remain so into the future.  (Stability at earlier times is guaranteed by well-posedness, albeit with poor quantitative dependence.)  Taking this approach one step further, the paper \cite{KenMart2009} proves \emph{asymptotic} stability of multisolitons.  Concretely, initial data that is a small $H^{1/2}(\R)$ perturbation of a multisoliton is shown to converge to a multisoliton (with nearby parameters) in $L^2$ sense \emph{on windows traveling with the constituent solitons}.  The conservative nature of the equation forbids convergence globally in space.  The point is that solitons travel at distinct speeds and in the opposite direction to radiation; thus, they should eventually separate, allowing for convergence on windows traveling with them.  Incidentally, an equicontinuity argument that we will present later in this paper (see the proof of Theorem~\ref{T:Strong_loc_cv_L2}) permits one to upgrade the $L^2$-windowed convergence of \cite{KenMart2009} to the $H^{1/2}$ topology, thereby matching the regularity of their initial perturbation.

It is natural to compare the asymptotic stability of (multi)solitons with soliton resolution, which has been proved for many integrable models via inverse scattering technology.  This approach gives not only the means of determining the multisoliton component directly from the initial data but also detailed asymptotics on the radiational part of the solution.  On the other hand, existing implementations of this method have always required rather stringent spatial decay hypotheses, while works on stability have focused on $H^s(\R)$ metrics that better reflect the dynamical invariances of the equation.     In fact, once one removes the spatial decay hypotheses, one must fear embedded singular spectra, and the existing soliton+radiation dogma has no candidate for what long-time behaviour singular continuous spectrum may entail!   

Thus far, inverse scattering technology for \eqref{BO} has not yet reached a level of development that has allowed a proof of soliton resolution.  Nevertheless, by taking a very different approach based on the explicit formula of \cite{MR4662323}, such full long-time asymptotics were derived recently in the breakthrough work \cite{GGM26}.  This paper requires, inter alia, that the initial data $u_0$ satisfy $\langle x\rangle u_0(x) \in H^1(\R)$.

In this paper, we prove $L^2$-asymptotic stability of multisolitons under general $L^2(\R)$ perturbations.  Concretely,

\begin{theorem}[Asymptotic Stability]\label{T:Intro AS}
Fix $N\geq 1$ and a set of negative spectral parameters $\Lambda= \{\la_1< \cdots <\la_N\}$. There exists $\delta>0$ so that if the initial data $u_0 \in L^2(\R)$ satisfies
\begin{equation}\label{I:907}
\| u_0 - Q_{\Lambda, \vec c} \|_{L^2} < \delta \qtq{for some} \vec{c}= (c_1, \ldots, c_N)\in \R^N,
\end{equation}
then there exist new parameters $\tilde \Lambda,~\tilde c$ so that the solution $u(t)$ to \eqref{BO} with initial data $u(0)=u_0$ converges to the multisoliton solution $Q_{\tilde\Lambda,\, \tilde c(t)}$ on windows travelling with the constituent solitons:
\begin{align}\label{I:Main_Asymptotic_Limit}
\lim_{|t| \to \infty}\, \bigl\| u(t, x - 2t\tilde\la_j) - Q_{\tilde\Lambda,\, \tilde c(t)}(x - 2t\tilde \la_j) \bigr\|_{L^2_x([-R, R])} = 0
\end{align}
for each $1\leq j \leq N$ and every $R > 0$.\end{theorem}

For a more complete statement, see Theorem~\ref{T:Main_Asymptotic}.  There, one will also find how $\tilde\Lambda$ and $\tilde c$ are determined by the initial data $u_0$, as well as bounds on how far they can move from $\Lambda$ and $c$.

The further development of ideas from \cite{GGM26}, in the absence of additional spatial decay hypotheses, will play a major role in the proof of Theorem~\ref{T:Intro AS}.  In fact, we will be able to show that solitons are the \emph{only} coherent structures in the long-time asymptotics of $L^2(\R)$ solutions:

\begin{theorem}\label{T:Intro Strong_loc_cv_L2}
Let $u(t)$ denote the solution to \eqref{BO} with data $u(0)=u_0\in L^2(\R)$. Given $\la \in \R$, we have:\\
{\upshape (i)} If $\lambda$ is not eigenvalue of $L_{u_0}$, then for any $R>0$,
\begin{equation*}
\lim_{|t| \to \infty}\, \|  u(t, x - 2t \la) \|_{L^2([-R,R])} = 0.
\end{equation*}
{\upshape (ii)} If $\lambda$ is an eigenvalue of $L_{u_0}$ and $f$ is an $L^2$-normalized eigenfunction associated with $\la$, then for any $R>0$,
\begin{equation*}
\lim_{|t| \to \infty} \, \left\| u(t, x - 2t \la) - \tfrac{4|\la|}{4\lambda^2(x-c)^2 + 1 } \right\|_{L^2([-R,R])} = 0, \qtq{where} c: = \Re\langle f, Xf \rangle 
\end{equation*}
and the operator $X$ is as defined on \eqref{X0}.
\end{theorem}

Although embedded eigenvalues can be ruled out in the $L^2(\R)$ setting (cf. Proposition~\ref{P:Wu_identity}), it is currently unknown whether embedded singular continuous spectrum can occur and how this might affect the long-time behaviour. Theorem~\ref{T:Intro Strong_loc_cv_L2} indicates that any such embedded singular continuous spectrum can only lead to structures that decay as $t\to \pm \infty$.

Let us now give a brief overview of the arguments used to prove Theorems~\ref{T:Intro AS} and \ref{T:Intro Strong_loc_cv_L2}, along with how this is reflected in the arrangement of the paper.

In Section~\ref{S:2} we review the key notations.  We also prove a commutator bound in Lemma~\ref{L:weights} that will be important for proving weighted estimates later.

Section~\ref{S:3} is divided into two subsections.  The first discusses the Lax operator $L_u$ for general $u\in L^2(\R)$; this includes detailing the Sobolev spaces adapted to this operator, the dependence of its resolvent on the potential, and its interrelations with the generator $X$ of Fourier translations.

Subsection~\ref{SS:3.2} discusses the Peter operator $P_u$ and the family of unitary flows $U(t)$ defined by $\partial_t U(t) = P_{u(t)} U(t)$.  As the construction of $U(t)$ has been omitted in the existing literature, it is tempting to underestimate it.  However, our Proposition~\ref{P:Unitary_Propagator} is actually intimately intertwined with the hard-won result that \eqref{BO} is globally well-posed in $L^2(\R)$!  This is because for our choice of Peter operator (which is not the most popular), the unitary flow actually encapsulates the \eqref{BO} flow: $u_+(t) = U(t)u_+(0)$.

Section~\ref{S:EFG} reviews the explicit formula \eqref{EFG} of G\'erard, as well as proving the extension \eqref{GEFG}.  The utility of such extensions in understanding long-time behavior was recently compellingly demonstrated in \cite{GGM26}.

Past experience has shown that (multi)solitons are intimately connected with eigenvalues $\lambda$ and eigenvectors $f$ of the Lax operator.  Thus, to understand perturbations of (multi)solitons, we must control how changes to the potential $u$ affect both $\lambda$ and $f$; see Proposition~\ref{P:Spectral_bounds}.  This is the first of three milestones to be attained in Section~\ref{S:profiles}.

The second milestone in Section~\ref{S:profiles} is understanding the long-time asymptotics along spacetime rays in the weak topology.  Working in this topology amounts to proving convergence not of $u(t,x)$ with $x\in \R$, but of its extension to the upper half-plane $\C_+$.  This step relies on an elaboration of the methods used in \cite{GGM26} to understand the solitonic component.  However, we are able to treat not only negative speeds, which one naturally associates with solitons and eigenvalues, but \emph{all} speeds.  This is important for fulfilling our goal of showing that there are no new coherent structures in this regime.  Of course, we also work under much fewer hypotheses than \cite{GGM26}.

The third component of Section~\ref{S:profiles} is upgrading from weak convergence to strong convergence, which we accomplish using equicontinuity. The climax of Section~\ref{S:profiles} is the proof of Theorem~\ref{T:Intro Strong_loc_cv_L2}.  

Theorem~\ref{T:Intro Strong_loc_cv_L2} describes the long-time asymptotics for arbitrary initial data $u(0)\in L^2(\R)$; it makes no assumption of proximity to any (multi)soliton.  It is only in Section~\ref{S:Asymptotic_Stability} that we combine this result with the spectral bounds of Proposition~\ref{P:Spectral_bounds} to prove 
Theorem~\ref{T:Intro AS}, or rather, the more complete Theorem~\ref{T:Main_Asymptotic}.

\subsection*{Acknowledgements} R.~B. was supported by an AMS-Simons Travel Grant.  R.~K. was supported by NSF grant DMS-2452346 and the project ANR-25-CFFS-0004 ``PhysMathEDPInteg'' of the France 2030 program. M.~V. was supported by NSF grant DMS-2348018.

\section{Notations and preliminaries}\label{S:2}

In this paper, we adopt the following convention for the Fourier transform:
\begin{align}\label{FT}
\widehat f(\xi) = \frac{1}{\sqrt{2\pi}} \int_\R e^{-i\xi x} f(x)\,dx  \quad \text{and}\quad f(x) = \frac{1}{\sqrt{2\pi}} \int_\R e^{i\xi x} \widehat f(\xi)\,d\xi.
\end{align}
This is unitary under the inner product 
\[
    \langle f, g \rangle =\int_{\R} \overline{f(x)} g(x)\, dx.
\]
Additionally, we have $\widehat{fg}(\xi) = (2\pi)^{-\frac12} \widehat f *\widehat g.$

We write $A\lesssim B$ to indicate that $A\leq c B$ for some constant $c$.  In the case that the constant $c$ depends on external parameters, this will be indicated with subscripts.  We will write $A \approx B$ to indicate that $A\lesssim B\lesssim A$.

Given Banach spaces $X_1$ and $X_2$, we write $\mathcal{B}(X_1, X_2)$ for the space of bounded linear operators mapping from $X_1$ into $X_2$, endowed with the operator norm.

The Hilbert transform $H$ is defined via the principal value integral
\begin{equation}\label{H def}
    Hf(x) = \pv \frac{1}{\pi} \int \frac{f(y)}{x-y}\,dy \qtq{or equivalently,} \widehat{Hf}(\xi) = -i \operatorname{sgn}(\xi)\widehat{f}(\xi).
\end{equation}

The Hardy space is defined via
$$
L^2_+(\mathbb{R}) := \{ f \in L^2(\mathbb{R}) :\, \supp(\mkern1mu\widehat{f}\, ) \subseteq [0,+\infty) \}.
$$
We denote the Cauchy--Szeg\H{o} projection onto this space by
$$
	\CS = \tfrac{1}{2}(1+iH):L^2(\R)\to L_+^2(\R).
$$
The analytic part of a function $v$ is denoted by $v_+ := \CS v$; by comparison, 
$$
\CS v f := \CS(vf).
$$

\begin{lemma}\label{L:weights}
The family of weight functions
$$
\omega_\eps (x)= \frac{x+i}{\eps x+i} \qtq{satisfy} \|\partial_x [\omega_\eps, \Pi u] h\|_{L^2} \lesssim \sqrt\eps \|u\|_{L^2} \|h\|_{L^2},
$$
uniformly for $0<\eps\leq 1$, $u\in L^2(\R)$, and $h\in L^2_+(\R)$.
\end{lemma}

\begin{proof}
As 
\begin{align*}
[\omega_\eps, \Pi u] = [\omega_\eps, \Pi]u + \Pi [\omega_\eps,u] = \bigl[ \omega_\eps, \tfrac{1+iH}2\bigr] u =  \tfrac i2 [ \omega_\eps,H] u,
\end{align*}
a direct computation yields
\begin{align*}
\bigl( [\omega_\eps, \Pi u] f\bigr)(x)= -\frac{(1-\eps)}{2\pi(\eps x+i)} \int_\R \frac{u(y) f(y)}{\eps y+i}\, dy.
\end{align*}
Thus,
$$
\|\partial_x [\omega_\eps, \Pi u] f\|_{L^2} \lesssim (1-\eps)\eps \| (\eps x+ i)^{-2}\|_{L^2} \|u\|_{L^2} \|f\|_{L^2}\lesssim \sqrt\eps\|u\|_{L^2} \|f\|_{L^2},
$$
uniformly for $0<\eps\leq 1$, $u\in L^2(\R)$, and $f\in L^2_+(\R)$.
\end{proof}

\section{The Lax pair operators and their properties}\label{S:3}

One facet of the complete integrability of the Benjamin--Ono equation is that it admits a Lax pair representation
\begin{equation}
\text{u solves \eqref{BO}} \quad\iff\quad \partial_t L_{u(t)} = [P_{u(t)}, L_{u(t)}];\label{Lax eq}
\end{equation}
see \cite{Bock1979,Nakamura1979}. Such pairs were introduced by Peter Lax in the KdV setting.

The unbounded selfadjoint Lax operator $L_u$ and anti-selfadjoint Peter operator $P_u$, both acting on $L^2_+(\mathbb{R})$, are defined via
\begin{equation}\label{Lax op}
    L_u f := -i\partial_x f - \Pi(uf) \quad \text{and} \quad P_u f := -i\partial_x^2 f - 2\partial_x \Pi(uf) + 2 u'_+ f.
\end{equation}
Our choice of Peter operator in this paper is that of \cite{MR3484397}.  As pointed out in \cite{Killip2024}, one of the virtues of this choice of Peter operator is that solutions to \eqref{BO} satisfy
\begin{align}\label{u P eqn}
 \frac{d}{dt}\Pi u(t) = P_{u(t)} \Pi u(t).
\end{align}

In \cite{Killip2024}, it was shown that for any $u\in H^s(\R)$ with $s>-\frac12$, the operator $L_u$ can be realized as a semibounded selfadjoint operator with form domain $H^{1/2}(\R)$, and that its essential spectrum is $\sigma_{\mathrm{ess}}(L_u) = [0, \infty)$.

The spectral type of this operator was further elucidated by the authors in \cite{BKV25}, where the following generalization of the Wu identities was established:

\begin{prop}[Wu identities, \cite{BKV25}]\label{P:Wu_identity}
Fix $s>-\frac12$, $u\in H^s(\R)$, and let $f$ be an eigenfunction of $L_u$ associated with an eigenvalue $\lambda$.  Then $\lambda<0$, $\widehat{f}$ is continuous on $[0,+\infty)$,
\begin{equation}\label{Wu relation Hs}
    |\langle u, f \rangle|^2 = -2\pi \lambda \|f\|_{L^2}^2,
    \ \ \ \sqrt{2\pi} \la \widehat{f}(0)+\langle   u,f \rangle = 0, \ \ \ \text{and} \ \ \  \lambda|\widehat f(0)|^2 = -\|f\|^2_{L^2}.
\end{equation} 
\end{prop}

The first identity in \eqref{Wu relation Hs} was originally derived by Wu \cite{MR3484397} for $u\in L^\infty(\R)$ with $\langle x\rangle u\in L^2(\R)$. Subsequently, Sun \cite{Sun2021} removed the hypothesis $u \in L^\infty(\R)$.

Wu already observed that the first identity in \eqref{Wu relation Hs} forbade positive eigenvalues and guaranteed that negative eigenvalues are simple. Proposition~\ref{P:Wu_identity} extends this by addressing zero eigenvalues and by weakening the hypotheses on $u$.  Another direct application of the first identity is the trace bound
\begin{equation} \label{E:trace_ineq}
    \|u\|_{L^2}^2 \geq \sum_{\la_j\in\sigma_d(L_u)} |\langle u, f_j \rangle|^2  = 2\pi \sum_{\la_j\in\sigma_d(L_u)} |\lambda_j| ,
\end{equation}
where $f_j$ denote the $L^2$-normalized eigenfunctions of $L_u$ associated with the eigenvalues $\lambda_j$.

In the case of a multisoliton solution to \eqref{BO}, we have equality in \eqref{E:trace_ineq}, namely,
\begin{align}\label{soliton mass}
\bigl\|Q_{\Lambda, \vec c}\bigr\|_{L^2}^2 = 2\pi \sum_{\la\in\Lambda} |\lambda| \qtq{for any} \vec c\in \R^N.
\end{align}
This was proved by Sun, who gave the following thorough characterization:

\begin{theorem} [Spectral characterization of multisolitons, \cite{Sun2021}] \label{T: Sun}
Fix \( N\geq 1 \) and a set $\Lambda$ of negative parameters $\la_1<\cdots<\la_N$. A function $ u$ belongs to the set  $\{Q_{\Lambda, \vec c} :\, \vec c\in \R^N\}$ if and only if it satisfies the following three conditions:
\begin{enumerate}
    \item \( \langle x \rangle u \in L^2(\mathbb{R}) \).
    \item $\la_1, \ldots, \la_N$ are the only negative eigenvalues of the Lax operator \( L_u \). 
    \item \( \Pi u \) belongs to the span of the corresponding \( N \) eigenfunctions.
\end{enumerate}
\end{theorem}

\subsection{The Lax operator and properties of its eigenfunctions}\label{SS:3.1}

In this paper, we will need additional spectral information about the Lax operator and its eigenfunctions in the case $u\in L^2(\R)$.

The following theorem contains several key results we will rely on in our analysis, such as the identification of Sobolev norms adapted to the Lax operator and Lipschitz bounds satisfied by its resolvent.

\begin{theorem}[Lax operator]\label{T:Lu-Intro}
For any $u \in L^2(\mathbb{R})$, the Lax operator $L_u$ with domain $\mathcal{D}(L_u) = H^1_+(\mathbb{R})$ is selfadjoint and bounded from below, satisfying
\begin{equation}\label{L below}
L_u \geq -\tfrac{1}{2\pi} \|u\|_{L^2}^2 .
\end{equation}
Its essential spectrum is $\sigma_{\mathrm{ess}}(L_u) = [0, \infty)$.

Moreover, for all $\kappa \geq \kappa_0:= \frac2{\pi} \|u\|_{L^2}^2$, $\sigma \in [-1, 1]$, and $f\in H_+^{\sigma}(\R),$ we have
\begin{equation}
    \label{E:Equiv_norm}
    \| (L_u + \kappa)^\sigma f \|_{L^2} \approx \| (L_0 + \kappa)^\sigma f \|_{L^2}.
\end{equation}

Finally, for any $z\notin \sigma(L_u)$ and $v\in L^2(\R)$ such that
\begin{align}\label{close}
\|u-v\|_{L^2}\leq c(\|u\|_{L^2}) \frac{\operatorname{dist}\big(z, \sigma(L_u)\big)}{1+|z|}
\end{align}
for a sufficiently small constant $c(\|u\|_{L^2})$, we have that $z\notin \sigma(L_v)$,
\begin{align}
\big\|(L_v - z)^{-1}\big\|_{L^2_+\to H^1_+} &\lesssim_{\|u\|_2} \frac{1+|z|}{\operatorname{dist}\big(z, \sigma(L_u)\big)},\label{rez v}\\
\big\|(L_u - z)^{-1} - (L_v - z)^{-1}\big\|_{L^2_+\to H^1_+} &\lesssim_{\|u\|_2} \frac{(1+|z|)^2}{\operatorname{dist}^2\big(z, \sigma(L_u)\big)}\|u - v\|_{L^2}.\label{E:Lipschitz z}
\end{align}
\end{theorem}

\begin{proof}
Much can be understood about $L_u$ by treating it as a perturbation of $L_0$:
\begin{equation}\label{E:L[u]=L[0]+}
    (L_u + \kappa)f = (L_0 + \kappa)f - \Pi(uf).
\end{equation}
The operator $L_0 + \kappa$ is invertible on the Hardy space $L^2_+(\R)$ whenever $\kappa > 0$. For $h \in L^2_+(\R)$,
the H\"older inequality yields
\begin{align}\label{E: PIu(L0+k)}
    \|\Pi u(L_0+\kappa)^{-1}h\|_{L^2} \leq \|u\|_{L^2} \|(L_0+\kappa)^{-1}h\|_{L^\infty}
    &\leq \|u\|_{L^2} \tfrac{1}{\sqrt{2\pi}}\|\widehat{h}\|_{L^2} \bigl\|\tfrac{1}{\xi+\kappa}\bigr\|_{L^2}  \notag\\
    &\leq \tfrac{1}{\sqrt{2\pi}} \kappa^{-\frac12} \|u\|_{L^2} \|h\|_{L^2}.
\end{align}
In this way, we deduce that
\begin{equation}\label{KRinput}
    \| \Pi(uf) \|_{L^2} \leq \tfrac{1}{2} \|(L_0+\kappa)f\|_{L^2} \qtq{for all} \kappa\geq \kappa_0:=\tfrac2{\pi}\|u\|_{L^2}^2.
\end{equation}
By the Kato--Rellich Theorem, this guarantees that $L_u$ is selfadjoint with domain $H^1_+(\R)$.  Combining \eqref{KRinput} with \eqref{E:L[u]=L[0]+}, we also deduce 
\begin{equation*}
   \tfrac12 \|(L_0+\kappa)f\|_{L^2}\leq  \| (L_u + \kappa)f \|_{L^2}\leq \tfrac{3}{2} \|(L_0+\kappa)f\|_{L^2} \qtq{for all} \kappa\geq \kappa_0,
\end{equation*}
which proves \eqref{E:Equiv_norm} in the case $\sigma=1$.  As the case $\sigma = 0$ is trivial, the intermediate cases $\sigma \in (0, 1)$ follow by complex interpolation.  For negative powers $\sigma \in [-1, 0]$, the equivalence \eqref{E:Equiv_norm} may then be deduced from the standard duality arguments across the $L^2(\R)$ pairing, utilizing the selfadjointness of the operators involved. 

One can show that $L_u$ is a relatively compact perturbation of $L_0$ not only for $u\in L^2(\R)$, but even under weaker hypotheses; see \cite{Killip2024, Sun2021}.  Correspondingly, Weyl's Theorem guarantees that $\sigma_{\mathrm{ess}}(L_u) = [0, \infty)$.  In light of this, any negative spectrum must take the form of eigenvalues and so the lower bound \eqref{L below} follows from the first Wu identity \eqref{Wu relation Hs}.  Equality holds in the case of \eqref{1-soliton}, so this bound is sharp.
      
It remains to consider claims \eqref{rez v} and \eqref{E:Lipschitz z}.  For $z \notin \sigma(L_u)$, the spectral theorem for selfadjoint operators yields
\begin{align}\label{rez bdd}
\|(L_u - z)^{-1}\|_{L^2_+ \to L^2_+} \leq \frac{1}{\operatorname{dist}\big(z, \sigma(L_u) \big)}.
\end{align}
Using the resolvent identity to write
\begin{equation}\label{Rkz}
(L_u - z)^{-1}= (L_u +\kappa_0)^{-1} + (z+\kappa_0)(L_u +\kappa_0)^{-1}(L_u - z)^{-1}
\end{equation}
and employing \eqref{E:Equiv_norm} and \eqref{rez bdd}, we may bound
\begin{align}\label{H1 rez bdd}
\|(L_u - z)^{-1}\|_{L^2_+\to H^1_+}
&\leq \|(L_u +\kappa_0)^{-1}\|_{L^2_+\to H^1_+} \notag\\
&\quad+ (|z| + \kappa_0) \|(L_u +\kappa_0)^{-1}\|_{L^2_+\to H^1_+}\|(L_u - z)^{-1}\|_{L^2_+\to L^2_+} \notag\\
&\lesssim_{\|u\|_2} 1+\frac{1+|z|}{\operatorname{dist}\big(z, \sigma(L_u)\big)} \lesssim_{\|u\|_2} \frac{1+|z|}{\operatorname{dist}\big(z, \sigma(L_u)\big)}.
\end{align}
Consequently, for $z \notin \sigma(L_u)$ we obtain
\begin{align}\label{8:19}
\|\Pi(u-v)(L_u-z)^{-1}\|_{L^2_+\to L^2_+} 
&\leq \|u-v\|_{L^2}\|(L_u - z)^{-1}\|_{L^2_+\to H^1_+}\notag\\
&\lesssim_{\|u\|_2} \frac{1+|z|}{\operatorname{dist}\big(z, \sigma(L_u)\big)} \|u-v\|_{L^2}<1,
\end{align}
provided $v\in L^2(\R)$ satisfies \eqref{close} with a sufficiently small constant $c(\|u\|_{L^2})$.  This shows that for $z \notin \sigma(L_u)$ and $v$ satisfying \eqref{close}, the operator $L_v-z$ is invertible and its inverse satisfies
$$
(L_v-z)^{-1} = (L_u-z)^{-1} \sum_{\ell\geq 0} (-1)^\ell \bigl[\Pi(u-v)(L_u-z)^{-1}\bigr]^\ell.
$$
In particular, using \eqref{H1 rez bdd} and \eqref{8:19}, we get
\begin{align*}
\|(L_v- z)^{-1}\|_{L^2_+\to H^1_+}
&\lesssim \|(L_u - z)^{-1}\|_{L^2_+\to H^1_+}\sum_{\ell\geq 0}\|\Pi(u-v)(L_u-z)^{-1}\|_{L^2_+\to L^2_+}^\ell  \\
&\lesssim_{\|u\|_2} \frac{1+|z|}{\operatorname{dist}\big(z, \sigma(L_u)\big)},
\end{align*}
which proves \eqref{rez v}.

Finally, by the resolvent identity, \eqref{rez v}, and \eqref{H1 rez bdd}, we may bound
\begin{align*}
\|(L_u - z)^{-1} &- (L_v - z)^{-1}\|_{L^2_+\to H^1_+} \\
&= \|(L_u - z)^{-1} \Pi(u - v) (L_v - z)^{-1}\|_{L^2_+\to H^1_+}\\
&\leq \|(L_u- z)^{-1}\|_{L^2_+\to H^1_+} \|u-v\|_{L^2}\|(L_v- z)^{-1}\|_{L^2_+\to H^1_+}\\
&\lesssim_{\|u\|_2} \frac{(1+|z|)^2}{\operatorname{dist}^2\big(z, \sigma(L_u)\big)}\|u - v\|_{L^2},
\end{align*}
which completes the proof of \eqref{E:Lipschitz z}.
\end{proof}

We write $X$ for the (maximal) generator of the contraction semigroup
\begin{align}\label{X0}
\widehat{ e^{-itX} h }(\xi) = \widehat{h}(\xi+t) \qtq{acting on} L^2_+(\R).
\end{align}
Correspondingly,
\begin{align}\label{X2}
X h = \lim_{t\searrow0} \tfrac{1-e^{-itX}}{it} h \qtq{for all} h\in \D(X)
\end{align}
and
\begin{align*}
\D(X) := \bigl\{ h\in L^2_+(\R) : \text{this $L^2$-limit exists} \bigr\} = \bigl\{ h\in L^2_+(\R) : \widehat h \in H^1\bigl[0,\infty)\bigr)\bigr\}.
\end{align*}

We write $I_+$ for the unbounded linear functional
\begin{equation}\label{I_+}
   I_+ (f) := \lim_{y\to\infty} 2\pi y f(iy)  = \lim_{y\to\infty} \langle \chi_y, f\rangle \qtq{where} \chi_y(x):= \tfrac{iy}{x+iy}.
\end{equation}
The domain of $I_+$ is comprised of those functions for which the limit above exists.  We warn the reader that $I_+$ is \emph{not} closable on $L^2_+(\R)$; indeed, the kernel is dense.  The equivalent definition in Fourier variables reads 
\begin{equation}\label{I_+ hat}
   I_+ (f) := \lim_{y\to\infty} \sqrt{2\pi} \int_0^\infty y e^{-y\xi} \widehat f(\xi)\, d\xi.
\end{equation}
In particular, if $\widehat f$ is continuous on $[0,\infty)$, then $I_+(f) = \sqrt{2\pi} \widehat f(0)$.

Next, we recall some well-known commutation relations between $X$ and the Lax operator, extending them to the present setting of limited regularity and poor spatial decay:

\begin{lemma}\label{L:Com}
We have $[X,L_0]=i\Id$, by which we mean that for any $h\in L^2_+(\R)$,
\begin{align}\label{6.1}
 \bigl[ \tfrac{1 - e^{-itX}}{it} , \, L_0\bigr] h \longrightarrow i h \quad\text{in $L^2$ sense as $t\searrow 0$.}
\end{align}
Furthermore, if $u\in L^2(\R)$, then $[ X , \, \CS u \bigr] = \tfrac{i}{2\pi} u_+ \otimes I_+$ in the sense that if $h\in L^2_+(\R)$ is such that $\widehat h(\xi)$ is continuous on $[0,\infty)$, then
\begin{align}\label{6.3}
 \bigl[ \tfrac{1 - e^{-itX}}{it} , \, \CS u \bigr] h \longrightarrow \tfrac{i}{2\pi} I_+(h) u_+   \quad\text{in $L^2$ sense as $t\searrow 0$}.
\end{align}
\end{lemma}

\begin{proof}
We recall from \cite[Lemma 4.2]{BKV25} that for any $\sigma \in \R$ and $t \geq 0$,
$$
\|e^{-itX}f\|_{H^\sigma_+} \lesssim \langle t \rangle^{|\sigma|} \|f\|_{H^\sigma_+}.
$$
Consequently, the left-hand side of \eqref{6.1} is well-defined as a bounded mapping from $L^2_+(\R)$ into $H^{-1}_+(\R)$. However, a direct computation yields
\begin{equation}\label{6015}
    \tfrac{i}{t} [e^{-itX}, L_0] h = i e^{-itX} h,
\end{equation}
demonstrating that the commutator actually maps $L^2_+(\R)$ to $L^2_+(\R)$. Taking the limit as $t \searrow 0$, we find
$$
    \tfrac{i}{t}\bigl[ e^{-itX}, \, L_0\bigr] h\longrightarrow ih \quad\text{in $L^2_+(\R)$ as $t\searrow 0$}.
$$
A slightly longer computation (details available in the proof of \cite[Prop.~1.4]{BKV25}) shows that for $\xi\geq 0$,
\begin{align}\label{606}
\bigl( \tfrac{i}{t} \bigl[ e^{-itX}, \,\CS u \bigr] h \bigr)\widehat{\vphantom{\big|}\ }(\xi)
	&= \tfrac{i}{t\sqrt{2\pi}} \int_0^t \widehat u(\xi+t-\eta) \widehat h(\eta)\,d\eta .
\end{align}
Claim \eqref{6.3} follows easily from the continuity of $\widehat h$ and of translations in $L^2(\R)$.
\end{proof}

Our reason for revisiting these commutation relations is to understand how $X$ interacts with the resolvent of the Lax operator:

\begin{prop}\label{P:CT}
Fix $A>0$.  For any $u,v \in L^2(\mathbb{R})$ and $z\in \C$ satisfying
\begin{align}\label{6.7}
 \dist(z, \sigma(L_u)) > A^{-1}, \quad  \dist(z, \sigma(L_v)) > A^{-1}, \quad \|u\|_{L^2} + \|v\|_{L^2} + |z| < A,
\end{align}
we have 
\begin{align}\label{6.5}
\bigl\| X (L_u-z)^{-1} g \bigr\|_{H^1} &\lesssim_A \| g \|_{L^2} + \| X g \|_{L^2}, \\
	\label{6.6}
\bigl\| X \bigl[ (L_u-z)^{-1} - (L_v-z)^{-1}\bigr] g \bigr\|_{H^1} &\lesssim_A \Bigl[ \| g \|_{L^2} + \| X g \|_{L^2}\Bigr] \| u-v \|_{L^2},
\end{align}
uniformly for $g\in \D(X)$.
\end{prop}

\begin{proof}
We will focus first on \eqref{6.5} and then turn to \eqref{6.6} later.

By Theorem~\ref{T:Lu-Intro} and the resolvent identity, we know that
\begin{equation}\label{z iso}
\| f\|_{H^1_+} \lesssim_A \| (L_u-z) f \|_{L^2} \lesssim_A \| f\|_{H^1_+} .
\end{equation}

We will show that in the case of $h:= (L_u-z)^{-1} g$, the limit \eqref{X2} exists not only in $L^2(\R)$, but even in $H^1(\R)$.  In view of \eqref{z iso}, it suffices to show that  
\begin{align}\label{813}
(L_u-z) \tfrac{1-e^{-itX}}{it}  h \quad \text{converges in $L^2$ sense as $t\searrow 0$.}
\end{align}
Notice that by rearranging $(L_u-z) h=g$, we may write 
\begin{align}\label{814}
(L_u -z) \tfrac{1-e^{-itX}}{it} h = -\bigl[ \tfrac{1 - e^{-itX}}{it}, \, L_0\bigr] h + \bigl[ \tfrac{1 - e^{-itX}}{it}, \,\CS u \bigr] h
	+ \tfrac{1-e^{-itX}}{it} g .
\end{align}
As $g\in \D(X)$, the last term converges to $Xg$ as $t\searrow 0$.

Before we can send $t\searrow 0$ throughout \eqref{814} using Lemma~\ref{L:Com}, we first need to verify the continuity of $\widehat h(\xi)$.  For this purpose, we rewrite $(L_u-z) h=g$ as 
\begin{align*}
\widehat{h}(\xi) = \tfrac1{\xi-z} \bigl[ \mkern2mu \widehat{uh}(\xi) + \widehat g(\xi) \bigr] \qtq{for all} \xi\geq 0
\end{align*}
and then observe that 
\begin{align*}
uh\in L^1(\R) \qtq{with}  \| uh \|_{L^1} \leq \| u \|_{L^2} \| h \|_{L^2} \lesssim_A  \| g \|_{L^2},
\end{align*}
which guarantees that $\widehat{uh}(\xi)$ is bounded and continuous.  Also,
\begin{align*}
\| \widehat g \mkern 1mu\|_{C}^2 \leq 2 \| \widehat g\mkern 1.5mu '\|_{L^2} \| \widehat g \mkern 1mu\|_{L^2} \leq  2 \| X g\|_{L^2}   \| g \|_{L^2},
\end{align*}
so $\widehat g\mkern 1mu$ is also bounded and continuous. We may then infer that $\widehat{h}(\xi)$ is continuous on $[0,\infty)$ and
\begin{align*}
 \bigl| I_+(h) \bigr| \leq \|\widehat{h}\|_{L^\infty} \lesssim_A \bigl[ \| uh \|_{L^1} + \|\widehat g\mkern1.5mu \|_{L^\infty} \bigr]
 	\lesssim_A  \| g \|_{L^2} + \| X g \|_{L^2} .
\end{align*}

With this preliminary completed, we may now use Lemma~\ref{L:Com} to send $t\searrow 0$ in \eqref{814} to obtain \eqref{813} and the estimate \eqref{6.5}.

We turn now to \eqref{6.6}.  Given $w\in L^2(\R)$ and $g\in \D(X)$, we write
\begin{equation*}
\tfrac{1 - e^{-itX}}{it} \CS w (L_u-z)^{-1} g = \bigl[ \tfrac{1 - e^{-itX}}{it},\,\CS w \bigr] (L_u-z)^{-1}  g + \CS  w \tfrac{1 - e^{-itX}}{it} (L_u-z)^{-1}  g
\end{equation*}
with the intention of sending $t\searrow0$.  By \eqref{6.5} we know that $(L_u-z)^{-1}  g \in \D(X)$, which verifies the continuity requirement needed to apply \eqref{6.3} to the first term.  To take the limit in the second term, we further rely on the $H^1_+$ convergence result encapsulated in \eqref{813}; this guarantees $L^2_+$ convergence even after the application of $\CS w$.  In this way, we deduce first that $\CS w (L_u-z)^{-1} g \in \D(X)$ and then that
\begin{equation}\label{824}
\bigl\| X \CS w (L_u-z)^{-1} g \bigr\| \lesssim_A \| w\|_{L^2} \Bigl[ \| g \|_{L^2} + \| X g \|_{L^2}  \Bigr].
\end{equation}

To complete the proof of \eqref{6.6}, we use \eqref{824} with $w=u-v$, the identity
\begin{equation}
(L_u-z)^{-1} - (L_v-z)^{-1} = (L_v-z)^{-1} \CS (u-v) (L_u-z)^{-1}, 
\end{equation}
and the fact that \eqref{6.5} applies equally well to $L_v$.
\end{proof}

Our first application of this proposition is qualitative in nature, namely, proving that eigenfunctions of $L_u$ belong to the domain of $X$.  This property of eigenfunctions was observed previously in \cite{GG26}.  In Section~\ref{S:profiles} we will employ this more quantitatively in order to control how eigenfunctions depend on the potential $u\in L^2(\R)$.

\begin{corollary}\label{C: f in X}
For $u\in L^2(\R)$, any eigenfunction $f$ of the Lax operator $L_u$ satisfies $f\in \mathcal D(X)$ and $Xf \in H^1_+(\R)$.
\end{corollary}

\begin{proof}
By Proposition~\ref{P:Wu_identity}, all eigenvalues of $L_u$ are negative and simple.  Furthermore, as noted in Theorem~\ref{T:Lu-Intro}, the essential spectrum of $L_u$ is $[0,\infty)$.  Therefore, for any eigenvalue $\lambda$ of $L_u$, we may find a smooth contour $\Gamma$ in the left complex half-plane enclosing $\lambda$ but no other part of the spectrum.  Thus, the (rank-one) spectral projection onto the associated eigenspace can be expressed as
\begin{equation}\label{R to E}
E_\lambda := \tfrac{-1}{2\pi i} \oint_\Gamma  (L_u-z)^{-1} \,dz .
\end{equation}

Proposition~\ref{P:CT} guarantees that $E_\lambda g \in D(X)$ for any $g\in D(X)$; moreover,
\begin{align}\label{904}
\|X E_\lambda g\|_{H^1} \lesssim \|g\|_{L^2} + \|Xg\|_{L^2}.
\end{align}
To continue, we observe that $D(X)$ is dense in $L^2_+(\R)$; indeed, $D(X)$ contains every $g\in L^2_+(\R)$ for which $\widehat g \in C^\infty_c$.  Given an eigenfunction $f$, we choose such a $g$ with $\langle f, g\rangle\neq 0$; then $E_\lambda g = \langle f, g\rangle f$ belongs to $\D(X)$ and so $f\in \D(X)$. That $Xf\in H^1_+(\R)$ follows from \eqref{904}.
\end{proof}

\subsection{The Peter operator and its unitary flow}\label{SS:3.2}

The first result in this subsection contains several key estimates satisfied by the Peter operator $P_u$, as well as its commutation relation with $X$.

\begin{prop}[The Peter operator] For $u\in L^2(\R)$,
\begin{equation}\label{-iP} 
-i P_u : h \mapsto -\partial_x^2 h + 2i\partial_x \Pi(uh) - 2 i u'_+ h
\end{equation}
defines a semibounded selfadjoint operator with form domain $H^1_+(\R)$.  Moreover,
\begin{equation}\label{P is Lip}
\| P_{u} \|_{H^1_+ \to H^{-1}_+} \lesssim 1+ \| u \|_{L^2} \qtq{and} \| P_{u} - P_{v} \|_{H^1_+ \to H^{-1}_+} \lesssim \| u - v \|_{L^2}.
\end{equation}
Finally, $[X, P_u]  = -2 L_u $, by which we mean that for any $h\in H^1_+(\R)$ we have
\begin{equation}\label{E:XPu_commutator}
 \bigl[ \tfrac{1 - e^{-itX}}{it} , \, P_u \bigr] h \longrightarrow -2 L_u h  \quad\text{in $L^2$ sense as $t\searrow 0$}.
\end{equation}
\end{prop}

\begin{proof}
It is convenient to expand \eqref{-iP} out in the equivalent form
\begin{equation}\label{-iP'} 
-i P_u h = -\partial_x^2 h + i\partial_x \Pi u h + i \Pi u \partial_x h + 2 \Pi \Im (u'_+) h.
\end{equation}

Selfadjointness will follow from the KLMN Theorem, once we demonstrate that $-iP_u$ is an infinitesimally form bounded perturbation of $-iP_0=-\partial_x^2$.  For this purpose, it suffices to show that $-i(P_u-P_0)$ is a Hilbert--Schmidt operator from $H^1$ into $H^{-1}$; the proof of this fact will also demonstrate the estimates \eqref{P is Lip}.

Writing $\mathfrak I_2$ for the Hilbert--Schmidt class of mappings $L^2\to L^2$ and working in Fourier variables, we have 
\begin{align*}
\bigl\| \langle -i\partial_x\rangle^{-1} [P_u-P_0] \langle -i\partial_x\rangle^{-1} \bigr\|_{\mathfrak I_2}^2 
	\lesssim \int_0^\infty \int_0^\infty \frac{[\xi^2+\eta^2]|\widehat u(\xi-\eta)|^2\,d\eta\, d\xi}{\langle \xi \rangle^2\langle \eta\rangle^2}
	\lesssim \| u \|_{L^2}^2.
\end{align*}
The estimates \eqref{P is Lip} now follow easily; indeed, $P_u-P_v = P_{u-v} - P_0$.

It remains to establish \eqref{E:XPu_commutator}. Using \eqref{-iP'}, we may write
\begin{align}\label{1:59}
\bigl[ \tfrac{1 - e^{-itX}}{it} , \, P_u \bigr] 
= \tfrac{1}{t}\big[e^{-itX}, \, i P_u] 
&=  \tfrac{1}{t}\bigl[e^{-itX}, \, \partial_x^2\bigr] +  \tfrac{1}{t}\bigl[e^{-itX}, \, L_0 \Pi u\bigr] \notag\\
&\quad+ \tfrac{1}{t}\bigl[e^{-itX}, \, \Pi u L_0\bigr] - \tfrac{2}{t}\big[e^{-itX}, \, \Pi \Im (u_+') \bigr].
\end{align}
Although $e^{-itX} \partial_x^2$ and $\partial_x^2 e^{-itX}$ only map $H^1_+(\R)$ into $H^{-1}_+(\R)$, a straightforward computation reveals that the commutator behaves better; indeed, for any function $h\in H^1_+(\R)$,
\begin{align}\label{711}
\tfrac{1}{t}\bigl[e^{-itX}, \, \partial_x^2\bigr]h = t e^{-itX}h- 2 e^{-itX}L_0 h \longrightarrow -2L_0h \quad\text{in $L^2$ sense as $t\searrow 0$}.
\end{align}

To handle the second and third terms on the right-hand side of \eqref{1:59}, we use the commutator identity
$$
[A, BC] = [A,B]C + B[A,C]
$$
and \eqref{6015} to write
\begin{equation}
\begin{aligned}\label{712}
\tfrac{1}{t}\bigl[e^{-itX}, \, L_0 \Pi u\bigr] &= e^{-itX} \Pi u + \tfrac{1}{t} L_0 \bigl[e^{-itX}, \,  \Pi u\bigr],\\
\tfrac{1}{t}\bigl[e^{-itX}, \,  \Pi u L_0\bigr] &= \Pi u e^{-itX}  + \tfrac{1}{t}  \bigl[e^{-itX}, \,  \Pi u\bigr]L_0.
\end{aligned}
\end{equation}
For $h\in H^1_+(\R)$ we have
\begin{equation}
\begin{aligned}\label{713}
e^{-itX} \Pi u h+\Pi u e^{-itX}h &\longrightarrow 2\Pi u h \quad\text{in $L^2$ sense as $t\searrow 0$},\\
\tfrac{1}{t}  \bigl[e^{-itX}, \,  \Pi u\bigr] L_0 h &\longrightarrow 0  \quad\text{in $L^2$ sense as $t\searrow 0$},
\end{aligned}
\end{equation}
where we used \eqref{606} in the last line.

Using again \eqref{606}, for $\xi \geq 0$ we have
\begin{align*}
&\Bigl(\tfrac{1}{t} L_0 \bigl[e^{-itX}, \,  \Pi u\bigr]- \tfrac{2}{t}\big[e^{-itX}, \, \Pi \Im(u_+') \bigr] h\Bigr)\widehat{\vphantom{\big|}\ }(\xi)\\
&=\tfrac{1}{t\sqrt{2\pi}} \int_0^t \xi \, \widehat u(\xi+t-\eta) \widehat h(\eta)\,d\eta -\tfrac{2i}{t\sqrt{2\pi}} \int_0^t(\xi+t-\eta)\, \widehat{\Im u_+} (\xi+t-\eta) \widehat h(\eta)\,d\eta\\
&=\tfrac{1}{t\sqrt{2\pi}} \int_0^t \xi \, \widehat u (\xi+t-\eta) \widehat h(\eta)\,d\eta- \tfrac{1}{t\sqrt{2\pi}} \int_0^t (\xi+t-\eta) \, \widehat{u_+} (\xi+t-\eta) \widehat h(\eta)\,d\eta \\
	&\quad +\tfrac{1}{t\sqrt{2\pi}} \int_0^t (\xi+t-\eta) \, \widehat{u_-} (\xi+t-\eta) \widehat h(\eta)\,d\eta\\
&=- \tfrac{1}{t\sqrt{2\pi}} \int_0^t (t-\eta) \, \widehat{u} (\xi+t-\eta) \widehat h(\eta)\,d\eta,
\end{align*}
where the last step is based on Fourier support considerations.  Moreover,
$$
\tfrac{1}{t\sqrt{2\pi}} \int_0^t (t-\eta) \, \widehat{u} (\xi+t-\eta) \widehat h(\eta)\,d\eta \longrightarrow 0 \quad\text{in $L^2$ sense as $t\searrow 0$}.
$$
Finally, collecting \eqref{1:59}, \eqref{711}, \eqref{712}, and \eqref{713}, we deduce
\begin{equation*}
\bigl[ \tfrac{1 - e^{-itX}}{it} , \, P_u \bigr] h \longrightarrow -2L_u h \quad\text{in $L^2$ sense as $t\searrow 0$},
\end{equation*}
which settles \eqref{E:XPu_commutator}.
\end{proof}

As $P_u$ is anti-selfadjoint, it is natural to imagine that the equation $\frac{d}{dt} U = P_u U$ defines a family of unitary operators $U(t)$.  To make this rigorous, we begin by considering this problem for sufficiently smooth solutions $u(t)$ to \eqref{BO}.

\begin{prop}\label{U(t)_u_reg}
Let $u(t)$ be a global $H^3(\mathbb{R})$ solution to the Benjamin--Ono equation. For all $\psi_0 \in L^2_+(\mathbb{R})$, the initial-value problem
\begin{equation}\label{Initial_value_Pb_reg}
\frac{d}{dt}\psi(t) = P_{u(t)}\psi(t) \quad \text{with} \quad \psi(0) = \psi_0
\end{equation}
admits a unique global-in-time $C_t L^2_+$ solution. Moreover, for each $t \in \mathbb{R}$ the mapping $U(t) : \psi_0 \mapsto \psi(t)$ is unitary on $L^2_+(\R)$ and bounded on $H^2_+(\R)$. Furthermore,
\begin{equation}\label{LuU=ULu0}
\Pi u(t) = U(t)\Pi u(0) \quad \text{and} \quad L_{u(t)} U(t)= U(t)L_{u(0)}.
\end{equation}
\end{prop} 

\begin{proof}
We begin by recalling that an $H^\infty(\R)$ solution to \eqref{BO} satisfies
\begin{align}\label{908}
\|u\|_{L^\infty_tH^\sigma_x(\R\times\R)}\leq C(\|u(0)\|_{H^\sigma}) 
\end{align}
for all $\sigma\in \tfrac12 \N$. This follows from the conservation laws associated with the Benjamin--Ono equation and the argument of Lax \cite{Lax1968}. Correspondingly, an $H^3(\R)$ solution to \eqref{BO} satisfies \eqref{908} for $\sigma \in \{0,\frac12,1,\frac32,2,\frac52,3\}$.

As in \eqref{-iP'}, we recast $P_u$ as
$$
P_u = -i\partial_x^2 - (\partial_x \Pi u + \Pi u \partial_x) + 2i \Pi \Im(u'_+).
$$
It is easy to verify that
\begin{align}\label{P mapping}
\|P_u\|_{H^2_+\to L^2_+} = \|P_u\|_{L^2_+\to H^{-2}_+}\lesssim 1+ \|u\|_{H^1}.
\end{align}

To construct solutions to \eqref{Initial_value_Pb_reg}, we first consider the regularized initial-value problem
\begin{align}\label{510}
    \frac{d}{dt}\psi_\varepsilon(t) = P_{u(t)}^\varepsilon \psi_\varepsilon(t) \quad \text{with} \quad \psi_\varepsilon(0) = \psi_0,
\end{align}
where $0<\eps\leq 1$ and
$$
P_u^\eps:=  -i\partial_x^2 + N_{u}^\eps \qtq{with} N_u^\eps:= -\tfrac{\partial_x}{1-\eps^2\partial_x^2} \Pi u - \Pi u \tfrac{\partial_x}{1-\eps^2\partial_x^2}
	+ 2i \Pi \Im(u'_+).
$$

The operator $N_{u}^\varepsilon$ is bounded on $L^2_+(\mathbb{R})$; indeed,
\begin{align}\label{1133}
\|N_u^\eps\|_{L^2_+\to L^2_+} \lesssim \eps^{-1} \| u\|_{L^\infty} + \|u'\|_{L^\infty} \lesssim \eps^{-1} \| u\|_{H^2}.
\end{align}
Consequently, using Duhamel's formula
\begin{equation} \label{Duhamel eps}
    \psi_\varepsilon(t) = e^{-it\partial_x^2} \psi_0 + \int_{0}^t e^{-i(t-s)\partial_x^2} N_{u(s)}^\varepsilon \psi_\varepsilon(s) \, ds
\end{equation}
and employing contraction mapping, we can easily construct a unique solution $\psi_\eps\in C_t L^2_+([-T, T]\times\R)$ with $T=T(\eps,\|u(0)\|_{H^2})$.  In this step, we also made use of \eqref{908} with $\sigma=2$ to provide uniform-in-time control in \eqref{1133}.

To extend this local-in-time solution to a global-in-time solution, we first observe that the map
$$
  t\mapsto  e^{it\partial_x^2}\psi_\varepsilon(t) =  \psi_0 +\int_{0}^t  e^{is\partial_x^2} N_{u(s)}^\varepsilon \psi_\varepsilon(s) \, ds \qtq{is} C_t^1L_x^2((-T, T)\times\R)
$$
with
$$
\frac{d}{dt} \bigl[e^{it\partial_x^2}\psi_\varepsilon(t)\bigr] = e^{it\partial_x^2} N_{u(t)}^\varepsilon \psi_\varepsilon(t).
$$
This ensures that the $L^2$ norm of the solution $\psi_\eps$ is conserved in time:
$$
\frac{d}{dt}\| \psi_\varepsilon(t)\|_{L^2}^2 =\frac{d}{dt}\| e^{it\partial_x^2} \psi_\varepsilon(t)\|_{L^2}^2 =2\Re \langle  \psi_\varepsilon(t),  N_{u(t)}^\varepsilon \psi_\varepsilon(t)\rangle= 0,
$$
where in the last step we used that $N_u^\eps$ is anti-selfadjoint. The conservation of the $L^2$ norm makes it trivial to iterate the local well-posedness arguments to construct a \emph{global} solution to \eqref{510}.

The conservation of the $L^2$ norm also guarantees that the flow maps $U_\eps(t):=U_\varepsilon(t;0):\psi(0) \mapsto \psi_\eps(t)$ are isometries. Moreover, the uniqueness of solutions provided by contraction mapping implies that $U_\eps(t;0)U_\eps(0;t) = \operatorname{Id}$ for all $t\in \R$.  Therefore, $U_\eps(t)$ is unitary on $L^2_+(\R)$.

A standard Gronwall argument yields that the solution $\psi_\eps$ to \eqref{510} satisfies
\begin{equation}\label{Psi-H2-control}
\sup_{\eps>0}\, \sup_{|t|\leq T}\, \|\psi_\varepsilon(t)\|_{H^n} \leq C(\|u(0)\|_{H^3}, T)\|\psi_0\|_{H^n} \qtq{for} n=1,2,
\end{equation}
uniformly in $\eps>0$.  To complete this step, we require the solution $u$ to \eqref{BO} to be in $H^3(\R)$.

Next, we compare different regularizations, beginning with the case $\psi_0 \in H^2_+(\R)$. Noting that the operators $N^\eta_{u}$ and $N^\eps_u$ are anti-selfadjoint, we obtain
\begin{equation}\label{E:derivative_diff}
\frac{d}{dt} \|\psi_\eta(t) - \psi_\varepsilon(t)\|_{L^2}^2 = \operatorname{Re}\, \bigl\langle \psi_\eta(t) - \psi_\varepsilon(t), \big(N_{u(t)}^\eta - N_{u(t)}^\varepsilon\bigr) \bigl(\psi_\eta(t)+ \psi_\varepsilon(t) \bigr)\bigr\rangle
\end{equation}
for any $0 < \varepsilon < \eta$.  To bound the right-hand side here, we note that
\begin{align}\label{diff N}
\bigl\|\bigl[N_{u}^\eta - N_{u}^\varepsilon\bigr]f\bigr\|_{L^2}
&\lesssim \bigl\|\tfrac{(\eta^2-\eps^2)\partial_x^3}{(1-\eta^2\partial_x^2)(1-\eps^2\partial_x^2)} \Pi u f\bigr\|_{L^2} + \bigl\|\Pi u\tfrac{(\eta^2-\eps^2)\partial_x^3}{(1-\eta^2\partial_x^2)(1-\eps^2\partial_x^2)} f\bigr\|_{L^2}\notag\\
&\lesssim \eta \| u\|_{H^2} \|f\|_{H^2}.
\end{align}

For any $T>0$, using \eqref{diff N} and \eqref{Psi-H2-control} in \eqref{E:derivative_diff} we may bound
\begin{equation}\label{E:derivative_diff_2}
    \frac{d}{dt} \|\psi_\eta(t) - \psi_\varepsilon(t)\|_{L^2}\leq \eta C(\|u(0)\|_{H^3}, T)\| \psi_0\|_{H^2},
\end{equation}
uniformly for $|t|\leq T$. An application of the Gronwall inequality then yields that $\psi_\varepsilon$ converges in $C_tL^2_+([-T, T]\times\R)$ as $\varepsilon \to 0$ for any choice of $T > 0$ when $\psi_0 \in H^2_+(\R)$.

As $U_\varepsilon(t)$  is unitary and $H^2_+(\R)$ is dense in $L^2_+(\R)$, this convergence extends to every $\psi_0 \in L^2_+(\R)$. Taking $\eps\to0$ in \eqref{Duhamel eps} shows that the limiting function $\psi(t)$ satisfies the Duhamel formula corresponding to the initial-value problem \eqref{Initial_value_Pb_reg},
$$
\psi(t) = e^{-it\partial_x^2} \psi_0 + \int_{0}^t e^{-i(t-s)\partial_x^2} \bigl[ - (\partial_x \Pi u(s) + \Pi u(s) \partial_x)
	+ 2i \Pi \Im \bigl(u'_+(s)\bigr)\bigr] \psi(s) \, ds
$$
in $C_t H^{-1}_+$ sense.  Moreover, this solution inherits the conservation of $L^2$ norm from the regularized solutions $\psi_\eps$.

Sending $\varepsilon \to 0$ in $U_\eps(t;0)U_\eps(0;t) = \operatorname{Id}$ we find that the limiting operator $U(t)$ is unitary for all $t\in \R$. As the solution also inherits the persistence of regularity bounds \eqref{Psi-H2-control}, we see that $U(t)$ is bounded on $H^2_+(\R)$. Finally, uniqueness of the solution to \eqref{Initial_value_Pb_reg} follows directly from duality and the solvability of the time-reversed flow. 

 We now turn to \eqref{LuU=ULu0}.  The first identity in \eqref{LuU=ULu0} follows from \eqref{u P eqn} and the uniqueness of solutions to \eqref{Initial_value_Pb_reg}.
Moreover, in view of \eqref{P is Lip} and \eqref{P mapping}, for any $\varphi \in H^1_+(\mathbb{R})$ the function
$$
\psi(t)=\bigl[L_{u(t)} U(t) - U(t)L_{u(0)}\bigr]\varphi \in C_tL^2_+\cap C_t^1H^{-2}_+
$$
is a solution to the initial-value problem \eqref{Initial_value_Pb_reg} with $\psi(0) = 0$. The uniqueness of solutions to \eqref{Initial_value_Pb_reg} then also yields the second identity in \eqref{LuU=ULu0}.
\end{proof}

Our next result extends both the definition and the remarkable properties of the operator $U(t)$ from the setting of $H^3(\R)$ solutions to \eqref{BO} to that of merely $L^2(\R)$ solutions to \eqref{BO}.

\begin{prop}\label{P:Unitary_Propagator}
Fix $u_0 \in L^2(\mathbb{R})$ and let $u(t)$ denote the solution to \eqref{BO} with initial data $u(0)=u_0$. There exists a strongly continuous family of unitary operators $U(t)$ on the Hardy space $L^2_+(\mathbb{R})$ satisfying the following properties:\\[1mm]
{\upshape (i)} $U(0)=\operatorname{Id}$ and $\frac{d}{dt} U(t) = P_{u(t)}U(t)$ as operators from $H^1_+(\mathbb{R})$ to $H^{-1}_+(\mathbb{R})$.\\[1mm]
{\upshape (ii)} For all $t \in \mathbb{R}$ we have $L_{u(t)}U(t) = U(t)L_{u_0}$ as operators from $H^1_+(\mathbb{R})$ to $L^2_+(\mathbb{R})$.\\[1mm]
{\upshape (iii)} The solution to \eqref{BO} satisfies $ \Pi u(t) = U(t)\Pi u_0$.\\[1mm]
{\upshape (iv)} For any $s \in [-1, 1]$, the operator $U(t)$ maps $H^s_+(\mathbb{R})$ boundedly into $H^s_+(\mathbb{R})$.\\[1mm]
{\upshape (v)} If $v_n\to u_0$ in $L^2(\R)$, then the corresponding unitary operators satisfy $U_n(t)\to U(t)$ strongly on $L^2(\R)$.\\[1mm]
{\upshape (vi)} If $\psi_0 \in H^1_+(\R)$ and $\langle x\rangle \psi_0\in L^2(\R)$, then $\langle x\rangle U(t) \psi_0\in L^2(\R)$ for all $t\in \R$.
\end{prop}

\begin{proof}
Choose $u_{0,n}\in H^3(\mathbb{R})$ such that
\begin{align*}
u_{0,n} \longrightarrow u_0 \qtq{in $L^2(\mathbb{R})$} \qtq{with} \ \ \sup_{n\geq 1}\,\|u_{0,n}\|_{L^2}\leq2 \|u_0\|_{L^2}.
\end{align*}
Let $u_n$ denote the $H^3$ solutions to \eqref{BO} with data $u_n(0)=u_{0,n}$. The well-posedness theory for \eqref{BO} guarantees that $u_n\to u$ in $C_t^{ }L^2_x([-T,T]\times\R)$ for each $T>0$.

Applying Proposition~\ref{U(t)_u_reg} with $u=u_n$ for each $n\geq 1$, we find unitary operators $U_n(t)$ satisfying the corresponding properties (i) through (iv). 

We begin by establishing the \emph{uniform} boundedness of $U_n(t)$ on $H^1_+(\mathbb{R})$. As $U_n(t):H^1_+(\R) \to H^1_+(\R)$ and satisfy property (ii), we have
$$
(L_{u_n(t)} + \kappa) U_n(t) = U_n(t) (L_{u_{0,n}} + \kappa).
$$
Using \eqref{E:Equiv_norm} with $\kappa$ large, depending only on $\|u_0\|_{L^2}$, together with the unitarity of $U_n(t)$ on $L^2_+(\mathbb{R})$, for each $f \in H^1_+(\mathbb{R})$ we may bound
\begin{align*}
\| U_n & (t) f \|_{H^1}\\
&\lesssim \| (L_{0} + \kappa) U_n(t) f \|_{L^2}\approx \| (L_{u_n(t)} + \kappa) U_n(t) f \|_{L^2} = \| U_n(t) (L_{u_{0,n}} + \kappa) f \|_{L^2}\\
& \qquad\qquad \qquad \qquad \qquad\qquad \qquad \qquad\qquad \qquad \quad \ \,  = \| (L_{u_{0,n}} + \kappa) f \|_{L^2} \lesssim \|f\|_{H^1}
\end{align*}
uniformly for $t\in\R$ and $n\geq 1$. By the usual duality argument, we deduce
\begin{equation}\label{E:Un_uniformly_bounded_H1}
\sup_n \,\sup_{t\in\R}\, \bigl[\| U_n(t) \|_{H^1_+ \to H^1_+} + \| U_n(t) \|_{H^{-1}_+ \to H^{-1}_+ }\bigr] \leq C(\|u_0\|_{L^2}).
\end{equation}

Next, we prove that $U_n(t)$ is a Cauchy in $\mathcal{B}(H^1_+(\R), H^{-1}_+(\R))$. Differentiating the operator product $U_m(t)^* U_n(t)$ and integrating from $0$ to $t$, we find
\begin{equation*}
    U_n(t) - U_m(t) = \int_0^t U_m(t) U_m(s)^* \bigl[ P_{u_n(s)} - P_{u_m(s)} \bigr] U_n(s) \, ds.
\end{equation*}
Using \eqref{P is Lip} and \eqref{E:Un_uniformly_bounded_H1}, we may bound
\begin{align*}
&\bigl\| U_m(t)U_m (s)^*[P_{u_n(s)} - P_{u_m(s)}] U_n(s)\bigr \|_{H^{1}_+ \to H^{-1}_+} \\
&\le\| U_m(t) \|_{H^{-1}_+ \to H^{-1}_+} \| U_m(s)^* \|_{H^{-1}_+ \to H^{-1}_+} \| P_{u_n(s)} - P_{u_m(s)} \|_{H^1_+ \to H^{-1}_+} \| U_n(s) \|_{H^1_+ \to H^1_+} \\
& \lesssim \| u_n (s)- u_m(s) \|_{L^2},
\end{align*}
where the implicit constant depends only on the $L^2(\R)$ norm of $u_0$. Consequently, for each $T>0$ and $|t|\leq T$ we have
\begin{align*}
\| U_n(t) - U_m(t) \|_{H^1_+ \to H^{-1}_+} &\lesssim T\, \|u_n-u_m\|_{L^\infty_t L_x^2([-T,T]\times\R)} \longrightarrow 0 \qtq{as} n,m\to \infty.
\end{align*}

Let $U(t)$ denote the limit of $U_n(t)$ in $\mathcal{B}(H^1_+(\R), H^{-1}_+(\R))$.  Using \eqref{E:Un_uniformly_bounded_H1}, we see that $U_n(t)f\rightharpoonup U(t)f$ in $H^1_+(\R)$ for each $f\in H^1_+(\R)$, as well as
\begin{align}\label{1228}
\sup_{t\in\R}\, \bigl[\| U(t) \|_{H^1_+ \to H^1_+} + \| U(t) \|_{H^{-1}_+ \to H^{-1}_+ }\bigr] \leq C(\|u_0\|_{L^2}).
\end{align}
Thus, taking $n\to \infty$ in the identity
\[
U_n(t)=\int_0^tP_{u_n(s)} U_n(s) \, ds,
\]
and invoking \eqref{P is Lip}, we find
$$
U(t)=\int_0^tP_{u(s)} U(s) \,ds \qtq{in}  \mathcal{B}(H^1_+(\R), H^{-1}_+(\R)),
$$
which proves property (i).

Next we prove that $U(t)$ is unitary on $L^2_+(\R)$. For any $f \in H^1_+(\R)$, we write
\begin{equation*}
    \|U(t)f\|_{L^2}^2 - \|U_n(t)f\|_{L^2}^2 = \langle (U(t)-U_n(t))f, U(t)f \rangle + \langle U_n(t)f, (U(t)-U_n(t))f \rangle.
\end{equation*}
Using \eqref{E:Un_uniformly_bounded_H1} and \eqref{1228}, we may thus bound
\begin{align*}
\bigl|\|U(t)f\|_{L^2}^2 &- \|U_n(t)f\|_{L^2}^2 \bigr|\\
&\le \| U(t)\!-U_n(t) \|_{H^{-1}_+ \to H^1_+}\bigl[\| U_n(t) \|_{H^1_+\to H^1_+} \! +\| U(t) \|_{H^1_+\to H^1_+} \bigr]  \|f \|_{H^{1}}^2,
\end{align*}
which converges to zero as $n\to \infty$. As $U_n(t)$ is unitary on $L^2_+(\R)$, we deduce that
$$
\|U(t)f\|_{L^2} = \|f\|_{L^2} \qtq{for all} f \in H^1_+(\R).
$$
Using this and the density of $H^1_+(\R)$ in $L^2_+(\R)$, we see that $U(t)$ extends to an isometry on $L^2_+(\R)$.  Employing a parallel argument for $U_n(t)^*$ yields that $U(t)$ is unitary on $L^2_+(\R)$.

Next, using that $U_n(t)$ converges to $U(t)$ in $\mathcal{B}(H^1_+(\R), H^{-1}_+(\R))$ together with \eqref{E:Un_uniformly_bounded_H1}, we see that for all $f\in H^1_+(\R)$ we have 
\begin{align}\label{107}
U_n(t) f\longrightarrow U(t)f \qtq{in} L^2_+(\R) \qtq{as} n\to \infty.
\end{align}
Combining this with the density of $H^1_+(\R)$ in $L^2_+(\R)$ and the unitarity of $U_n(t)$ and $U(t)$ on $L^2_+(\R)$, we find 
\begin{align}\label{108}
U_n(t) f\longrightarrow U(t)f \qtq{in} L^2_+(\R) \qtq{as} n\to \infty
\end{align}
for any $f\in L^2_+(\R)$.

Employing \eqref{107} and \eqref{108} and taking $n\to \infty$ in the identities
$$
L_{u_n(t)} U_n(t) = U_n(t) L_{u_{0,n}}  \qtq{and} \Pi u_n(t) = U_n(t) \Pi u_{0,n}
$$
yields properties (ii) and (iii).  Note that the convergence of the Lax operators is guaranteed by \eqref{E:Lipschitz z}.  Property (iv) follows by interpolation between the unitarity of $U(t)$ on $L^2_+(\R)$ and the bounds in \eqref{1228}. Property (v) is a direct consequence of our construction of $U(t)$ on $L^2_+(\R)$ via smooth approximation; see~\eqref{108}.

It remains to prove property (vi).  To this end, we fix $\psi_0\in H^1_+(\R)$ satisfying $\langle x\rangle \psi_0\in L^2(\R)$.  As $U(t)$ is bounded on $H^1_+(\R)$, $\psi(t):=U(t)\psi_0$ satisfies
\begin{align}\label{124}
\sup_{|t|\leq T} \|\psi(t)\|_{H^1} \lesssim_T \|\psi_0\|_{H^1} \qtq{for each} T>0.
\end{align}

Given $0<\eps\leq 1$, we consider the family of weights
$$
\omega_\eps(x):= \tfrac{x+i}{\eps x+i}.
$$
Using the anti-selfadjointness of $P_u$, we compute
\begin{align*}
\tfrac{d}{dt} \|\omega_\eps \psi(t)\|_{L^2}^2 = 2 \Re \langle\omega_\eps \psi(t),  \omega_\eps P_{u(t)} \psi(t) \rangle  = 2\Re \langle\omega_\eps \psi(t),  [\omega_\eps, P_{u(t)}] \psi(t) \rangle.
\end{align*}
Employing \eqref{-iP}, we may write
\begin{align*}
 [\omega_\eps, P_{u}] =  [\omega_\eps, -i \partial_x^2] -  [\omega_\eps, 2\partial_x \Pi u]  
&= i( \omega_\eps''+ 2\omega_\eps'\partial_x) - 2 [\omega_\eps, \partial_x] \Pi u - 2\partial_x [\omega_\eps, \Pi u]\\
&= i( \omega_\eps''+ 2\omega_\eps'\partial_x) + 2 \omega_\eps' \Pi u - 2\partial_x [\omega_\eps, \Pi u].
\end{align*}
By Lemma~\ref{L:weights} and the conservation of $L^2(\R)$ norm by \eqref{BO} solutions, we may bound
\begin{align*}
\|[\omega_\eps, P_{u(t)}] \psi(t)\|_{L^2}\lesssim \bigl[\|\omega_\eps''\|_{L^\infty} \!+\! \|\omega_\eps'\|_{L^\infty} \!+\!1\bigr] \|u(t)\|_{L^2} \|\psi(t)\|_{H^1}\lesssim \|u(0)\|_{L^2} \|\psi(t)\|_{H^1}
\end{align*}
and so 
\begin{align*}
\tfrac{d}{dt} \|\omega_\eps \psi(t)\|_{L^2}^2 \lesssim \|\omega_\eps \psi(t)\|_{L^2} \|u(0)\|_{L^2} \|\psi(t)\|_{H^1}
\end{align*}
uniformly for $0<\eps\leq 1$. Combining this with \eqref{124}, we find
$$
\sup_{|t|\leq T} \|\omega_\eps \psi(t)\|_{L^2} \lesssim \|\langle x\rangle \psi_0\|_{L^2} + T\|u(0)\|_{L^2}\|\psi_0\|_{H^1}, \qtq{uniformly for } 0<\eps\leq 1.
$$
Property (vi) follows by sending $\eps\to 0$.
\end{proof}

\section{The explicit formula}\label{S:EFG}

The explicit formula for \eqref{BO} reads as follows:
\begin{equation}\label{EFG}
   u_+(t, z) = \tfrac{1}{2\pi i} I_+ \Bigl( (X - 2t L_{u_0} - z)^{-1} \Pi u_{0} \Bigr), \qtq{for all} z \in \C_+.
\end{equation}
This was first discovered by G\'erard in \cite{MR4662323}, where it was shown to hold for solutions belonging to $H^2(\R)$.
Subsequently, the case of $L^2(\R)$ solutions was discussed in \cite{chen2025explicit}.

Following \cite{GGM26}, we will need to consider what becomes of the explicit formula when one replaces the vector $\Pi u_0$ by something else, while still retaining the Lax operator associated to $u_0$.  This already requires us to revisit the derivation of this formula.  Secondly, we employ a different Peter operator than \cite{chen2025explicit, GGM26,MR4662323}; this leads to a different unitary flow and so a different generalization of \eqref{EFG}.

Naturally, the first step is to understand the operator $X - 2t L_{u}$, starting with the case $u\equiv 0$.  As observed in \cite{KLV2025}, for example,
$X - 2tL_0$ is the generator of the contraction semigroup
\begin{equation}\label{}
s\mapsto e^{-is(X - 2tL_0)} =   e^{it\partial^2} e^{-isX} e^{-it\partial^2} 
\end{equation}
and correspondingly,
\begin{equation}\label{A0 4}
\D(X - 2tL_0) = \{ h \in L^2_+(\R) : e^{it\xi^2}\widehat h(\xi) \in H^1([0,\infty))\}.
\end{equation}
A more significant result from \cite{KLV2025} is the estimate
\begin{align}\label{A0 5}
\bigl\| (X-2tL_0-z)^{-1} \bigr\|_{L^2_+ \to L^\infty_+} &\lesssim (|t| \Im z)^{-\frac12}, \quad\text{for $\Im z>0$ and $t\neq0$,}
\end{align} 
which will allow us to understand $X - 2t L_{u}$ as a maximally accretive operator and prove the conjugation property \eqref{A0U}:

\begin{lemma}\label{L:A0U}
For any $u\in L^2(\R)$ and $t\in\R$, the operator $X-2tL_{u}$ is maximally accretive on the domain \eqref{A0 4}.  Moreover, if
$u(t)$ is a solution to \eqref{BO} with initial data $u(0)\in L^2(\R)$ and $U(t)$ denotes the family of unitary operators constructed in Proposition~\ref{P:Unitary_Propagator}, then
\begin{gather}\label{A0U}
U(t) : \D(X - 2tL_0) \to \D(X) \qtq{and} X U(t) = U(t) \bigl(X-2tL_{u(0)}\bigr) . 
\end{gather}
\end{lemma}

\begin{proof}
From \eqref{A0 5} and H\"older's inequality, we obtain 
\begin{align}\label{A0 6}
\bigl\| 2 t \CS u (X-2tL_0-z)^{-1} \bigr\|_{L^2_+ \to L^2_+} & \lesssim |t|^{1/2} (\Im z)^{-\frac12} \| u\|_{L^2} .
\end{align}
This allows one to apply the Kato--Rellich technique to deduce that $X-2tL_u$ shares the domain \eqref{A0 4}, is maximally accretive,  and 
\begin{align}\label{A0 7}
\bigl\| (X-2tL_u-z) h \bigr\|_{L^2_+} &\approx \bigl\| (X-2tL_0-z) h\bigr\|_{L^2_+} \qtq{for} \Im z \gtrsim |t|\,\|u\|_{L^2}^2.
\end{align}

We turn now to the interactions of this operator with $U(t)$.  Initially, we assume that $u(0)\in H^3_+(\R)$ in order to take advantage of Proposition~\ref{U(t)_u_reg}. In particular, for any Schwartz function $\psi\in L^2_+(\R)$, we have $\psi(t)=U(t)\psi \in H^2_+(\R)$.  Using also property (vi) in Proposition~\ref{P:Unitary_Propagator}, we have $\langle x\rangle \psi(t) \in L^2(\R)$ for all $t\in\R$.  Choosing two such Schwartz functions $\phi,\psi\in L^2_+(\R)$, we may use that $P_u$ is anti-selfadjoint, \eqref{E:XPu_commutator}, and \eqref{LuU=ULu0} to compute
\begin{align*}
\tfrac{d\ }{dt} \bigl\langle U(t)\phi,\ X U(t)\psi\bigr\rangle &= \bigl\langle U(t)\phi,\ \bigl[X,\, P_{u(t)}\bigr]U(t)\psi\bigr\rangle \\
	&= -2 \bigl\langle U(t)\phi,\ L_{u(t)} U(t)\psi\bigr\rangle  = -2 \bigl\langle\phi,\ L_{u(0)} \psi\bigr\rangle.
\end{align*}
In this way, we see that
\begin{align}\label{A0tmp}
\bigl\langle U(t)\phi,&\ XU(t)\psi\bigr\rangle =  \bigl\langle \phi,\ \bigl(X - 2t L_{u(0)}\bigr)\psi\bigr\rangle
\end{align}
for every pair of Schwartz functions $\phi,\psi\in L^2_+(\R)$.  It is now possible to remove the restriction $u(0)\in H^3(\R)$ by approximation and Proposition~\ref{P:Unitary_Propagator}.  Lastly, the full assertion \eqref{A0U} follows from this, the unitarity of $U(t)$ and, crucially, the fact  that both $X-2tL_{u(t)}$ and $X$ are closed operators.
\end{proof}

From \eqref{A0U}, we see that
\begin{equation}\label{almost EFG}
(X-z)^{-1} U(t) = U(t)  \bigl(X-2tL_{u(0)}-z\bigr)^{-1} 
\end{equation}
and that both operators map into $\D(X)$ and so into the domain of $I_+$. This brings us quite close to verifying \eqref{EFG}, especially when combined with the Cauchy integral formula:
\begin{equation}\label{CIF}
h(z)  = \tfrac{1}{2\pi i} I_+\bigl( (X-z)^{-1} h \bigr) \qtq{for every} z\in\C_+\qtq{and} h \in L^2_+(\R).
\end{equation}
The one task that remains is to verify that $I_+ \circ U(t) = I_+$ in a suitable sense.  For this purpose, we mirror the analysis in \cite[Lemma~4.2]{KLV2025}:

\begin{lemma}\label{L:chi}
For $y>0$ let $\chi_y(x) := \tfrac{iy}{x+iy}$ be as in \eqref{I_+}. Let $U(t)$ denote the family of unitary operators associated to a solution $u(t)$ to \eqref{BO} with initial data $u_0\in L^2(\R)$.  Then
\begin{gather}\label{U chi} 
 \| U(t)^*\chi_y - \chi_y\|_{L^2_+} \lesssim y^{-\frac32} |t| + y^{-1} \,|t| \,\| u_0 \|_{L^2} \qtq{for all} t\in \R.
\end{gather}
\end{lemma}

\begin{proof}
We will apply the fundamental theorem of calculus. To this end, we note~that 
\begin{equation*}
\tfrac{d}{dt} U(t)^*\chi_y = - U(t)^*P_{u(t)}\chi_y, \quad - P_u \chi_y = i\chi''_y + 2 \CS( u_-'\chi_y) + 2 \CS (u\chi_y') ,
\end{equation*}
as well as
\begin{align*}
\| \chi''_y \|_{L^2_+} \lesssim y^{-\frac32} \qtq{and}
	\| \CS (u\chi_y') \|_{L^2_+} \leq \| \chi_y' \|_{L^\infty} \| u \|_{L^2}  \lesssim y^{-1} \| u \|_{L^2}.
\end{align*}
To treat the middle term, we first write
\begin{align*}
\widehat{ u_-'\chi_y }(\xi) = \int_\xi^\infty i(\xi-\eta) \widehat u(\xi-\eta) \,y e^{-y\eta}\,d\eta,
\end{align*}
valid for all $\xi\geq 0$, and so deduce that
\begin{align*}
\| \CS( u_-'\chi_y) \|_{L^2_+} \lesssim \int_0^\infty \| \widehat u\|_{L^2} \,\eta y e^{-y\eta}\,d\eta \lesssim y^{-1} \| u \|_{L^2}.
\end{align*}
Thus \eqref{U chi} follows by integrating in time and using the $L^2$ conservation law for \eqref{BO}.
\end{proof}

\begin{prop}\label{P:EFG}
Let $u(t)$ be the solution to \eqref{BO} with initial data $u_0\in L^2(\R)$ and let $U(t)$ denote the associated family of unitary operators.  Then,
\begin{equation}\label{GEFG}
   \tfrac{1}{2\pi i} I_+ \Bigl( (X - 2t L_{u_0} - z)^{-1} h \Bigr) = [U(t) h](z)
\end{equation}
for all $z \in \C_+$ and $h\in L^2_+(\R)$. In particular, \eqref{EFG} holds and
\begin{equation}\label{BEFG}
    \left| I_+ \Bigl( (X - 2t L_{u_0} - z)^{-1} h \Bigr) \right| \leq (4\pi \Im z)^{-\frac12} \|h\|_{L^2_+}.
\end{equation}
\end{prop}

\begin{proof}
Applying \eqref{I_+}, Lemma~\ref{L:chi}, \eqref{almost EFG}, and then \eqref{CIF}, we find
\begin{align*}
   I_+ \Bigl( (X - 2t L_{u_0} - z)^{-1} h \Bigr) &= \lim_{y\to\infty} \bigl\langle \chi_y,\ (X - 2t L_{u_0} - z)^{-1} h\bigr\rangle\\
   &= \lim_{y\to\infty} \bigl\langle U^*(t) \chi_y,\ (X - 2t L_{u_0} - z)^{-1} h\bigr\rangle \\
   &= \lim_{y\to\infty} \bigl\langle \chi_y,\ (X-z)^{-1} U(t) h\bigr\rangle \\
   &= 2\pi i [U(t)h](z).
\end{align*}
This proves \eqref{GEFG}.  Choosing $h=\CS u_0$, the explicit formula \eqref{EFG} follows from part~(iii) of Proposition~\ref{P:Unitary_Propagator}.

Lastly, the Riesz representative of the linear functional $h\mapsto h(z)$ is
$$
\tfrac1{2\pi i}(x-\overline z)^{-1}, \qtq{which has $L^2$ norm}  (4\pi \Im z)^{-\frac12}.
$$
Thus, \eqref{BEFG} follows from \eqref{GEFG} and the unitarity of $U(t)$.
\end{proof}

\section{Asymptotic Profiles}\label{S:profiles}

To extract the large-time behavior of the solution from the explicit formula \eqref{EFG}, we must analyze the spectral components of the Lax operator.

By Theorem~\ref{T: Sun}, the spectral data of an $N$-soliton solution is encoded by a set $\Lambda$ of negative parameters $\la_1< \cdots < \la_N$. However, an arbitrarily small perturbation of an $N$-soliton solution may generate a countably infinite number of additional discrete eigenvalues. Our first goal in this section is to show a form of spectral stability for the Lax operator. Specifically, we will show that the set of the largest (in absolute value) $N$ eigenvalues remains robust, while the `spectral dust' formed by the additionally generated eigenvalues is strictly controlled by the size of the perturbation.

\begin{prop}[Spectral perturbation bounds]\label{P:Spectral_bounds}
Fix an integer $N\geq 1$ and a set $\Lambda$ of negative parameters $\la_1< \cdots < \la_N$. There exists $\delta > 0$ such that for any $u_0\in L^2(\R)$ satisfying
\begin{align}\label{935}
 \|u_0 - Q_{\Lambda, \vec{c}}\|_{L^2} < \delta \qtq{for some} \vec c= (c_1, \ldots, c_N)\in \R^N,
\end{align}
the Lax operator $L_{u_0}$ has $N$ negative eigenvalues $\tilde \lambda_1< \tilde \lambda_2 < \dots < \tilde \lambda_N$ satisfying
\begin{equation}\label{E:continuity_lambda}
\sup_{1\leq j\leq N}\, |\tilde \lambda_j - \lambda_j| \lesssim \|u_0 - Q_{\Lambda, \vec{c}}\|_{L^2}.
\end{equation}
Moreover, the remaining discrete spectrum satisfies 
\begin{equation}\label{E:trace}
    \sum_{\tilde \lambda_k \in \sigma_d(L_{u_0}) \setminus \{\tilde \lambda_1, \dots, \tilde\lambda_N\}} |\tilde \lambda_k| \lesssim \|u_0 - Q_{\Lambda, \vec{c}}\|_{L^2}.
\end{equation}
Furthermore, if we let $f_j, \tilde f_j$ denote $L^2$-normalized eigenvectors corresponding to the eigenvalues $\lambda_j, \tilde \lambda_j$ of $L_{Q_{\Lambda, \vec{c}}}$ and $L_{u_0}$, respectively, then the centers $c_j:= \Re\langle f_j, Xf_j\rangle$ and $\tilde c_j:= \Re\langle \tilde f_j, X\tilde f_j\rangle$ satisfy
\begin{align}\label{centers close}
\sup_{1\leq j\leq N}\, |\tilde c_j - c_j| \lesssim \|u_0 - Q_{\Lambda, \vec{c}}\|_{L^2}.
\end{align}
\end{prop}

\begin{proof}
We start by proving \eqref{E:continuity_lambda}. By Theorem~\ref{T: Sun}, the Lax operator $L_{Q_{\Lambda, \vec c}}$ has exactly $N$ simple negative eigenvalues $\lambda_1 < \dots < \lambda_N$. Because these eigenvalues are distinct and bounded away from the essential spectrum $[0, \infty)$, we can choose small, mutually disjoint, closed contours $\Gamma_j$ in the open left half-plane $\{z\in\C\ :\ \Re(z)<0\}$ such that each $\Gamma_j$ encloses only $\lambda_j$.  Within this proof, implicit constants are permitted to depend on the precise choice of contours. 

For $0<\delta\leq 1$ sufficiently small so as to guarantee \eqref{close} with $u=u_0$, $v= Q_{\Lambda, \vec c}$, and $z\in \cup \mkern2mu\Gamma_j$, we may define the spectral projections
\begin{equation*}
\tilde E_j := E_j(u_0) := \tfrac{-1}{2\pi i} \oint_{\Gamma_j} (L_{u_0} - z)^{-1} \,dz \qtq{and} E_j:= E_j(Q_{\Lambda, \vec c}).
\end{equation*}
The Lipschitz continuity of the resolvent guaranteed by \eqref{E:Lipschitz z} yields
\begin{align}\label{1052}
    \| \tilde \lambda_j^n \tilde E_j - \lambda_j^nE_j\|_{L^2_+\to L^2_+}
    &\lesssim \oint_{\Gamma_j} |z|^n \bigl\| (L_{u_0} - z)^{-1} - (L_{Q_{\Lambda, \vec c}} - z)^{-1} \bigr\|_{L^2_+\to L^2_+} \,|dz| \notag\\
    &\lesssim_n \| u_0 - Q_{\Lambda, \vec c} \|_{L^2} \qtq{for all} n\in \Z.
\end{align}
Therefore, taking $n=0$ and $\delta$ even smaller if necessary, we get
\begin{align}\label{1053}
\| \tilde E_j - E_j \|_{L^2_+\to L^2_+}  \leq  \tfrac12 \qtq{and so} \operatorname{rank}(\tilde E_j) = \operatorname{rank}(E_j) = 1.
\end{align}
This implies that $L_{u_0}$ has exactly one simple eigenvalue $\tilde \lambda_j$ inside each contour $\Gamma_j$. The bound \eqref{E:continuity_lambda} on the eigenvalue shift follows from the $n=1$ case of \eqref{1052}:
\begin{equation*}
 |\tilde \lambda_j- \lambda_j| = \bigl|\tr \{ \tilde \lambda_j \tilde E_j - \lambda_j E_j\}\bigr|\leq 2 \| \tilde \lambda_j \tilde E_j - \lambda_j E_j\|_{L^2_+\to L^2_+}\lesssim \| u_0 - Q_{\Lambda, \vec c} \|_{L^2}.
\end{equation*}
(Recall that the trace is bounded by the rank times the operator norm.)

In view of \eqref{1052} and \eqref{1053}, we have
\begin{align}\label{1054}
\|\tilde E_j f_j - f_j\|_{L^2}\lesssim  \| u_0 - Q_{\Lambda, \vec c} \|_{L^2} \qtq{and} \|\tilde E_j f_j\|_{L^2}\geq \tfrac12.
\end{align}
Choosing
\begin{align}\label{def tilde f}
\tilde f_j = \frac{\tilde E_j f_j}{\|\tilde E_j f_j\|_{L^2}},
\end{align}
we obtain
\begin{align}\label{eigen close}
\|\tilde f_j - f_j\|_{L^2}\leq \bigl\|\tilde f_j \|\tilde E_j f_j\|_{L^2} - f_j\bigr\|_{L^2} + \bigl| \|\tilde E_j f_j\|_{L^2}-1\bigr|\lesssim  \| u_0 - Q_{\Lambda, \vec c} \|_{L^2}.
\end{align}

We now turn to \eqref{E:trace}. To control the remaining discrete spectrum of $L_{u_0}$, we use \eqref{E:trace_ineq} which we restate as
\begin{equation}\label{E:trace_ineq2}
    2\pi \sum_{\tilde \lambda \in \sigma_d(L_{u_0})} |\tilde \lambda| \leq \| u_0 \|_{L^2}^2.
\end{equation}
On the other hand, by \eqref{soliton mass} we have
\begin{equation}\label{E:trace_eq_Q}
    2\pi \sum_{j=1}^N |\lambda_j| = \| Q_{\Lambda, \vec c} \|_{L^2}^2.
\end{equation}
Dividing the sum \eqref{E:trace_ineq2}  into the contribution of the $N$ primary eigenvalues and that of the remaining eigenvalues and using \eqref{E:trace_eq_Q}, we get
\begin{align*}
 2\pi \sum_{\tilde \lambda_k \notin \{\tilde\lambda_1, \dots, \tilde\lambda_N\}} |\tilde \lambda_k|
 &\leq \| u_0 \|_{L^2}^2 - 2\pi \sum_{j=1}^N |\tilde \lambda_j|\\
 &= \| u_0 \|_{L^2}^2 - \| Q_{\Lambda, \vec c} \|_{L^2}^2+ 2\pi \sum_{j=1}^N  |\lambda_j| - |\tilde\lambda_j| \\
&\leq \big| \| u_0 \|_{L^2}^2 - \| Q_{\Lambda, \vec c} \|_{L^2}^2 \big| + 2\pi \sum_{j=1}^N \big| \lambda_j - \tilde \lambda_j \big|.
\end{align*}
Applying \eqref{E:continuity_lambda}, we deduce \eqref{E:trace}.

Finally, we turn to \eqref{centers close}. Mirroring the analysis above and using \eqref{6.6}, we deduce
\begin{align}\label{155}
\|X(\tilde E_j - E_j)g\|_{L^2} \lesssim \| u_0 - Q_{\Lambda, \vec c}\|_{L^2} \bigl[ \|g\|_{L^2} + \|Xg\|_{L^2}\bigr] \qtq{for all} g\in D(X).
\end{align}
In particular, using also Corollary~\ref{C: f in X}, \eqref{1054}, and \eqref{def tilde f}, we have
\begin{align}\label{1055}
\tfrac12 \|X \tilde f_j\|_{L^2}\leq\|\tilde E_j f_j\|_{L^2} \|X \tilde f_j\|_{L^2} &=  \|X \tilde E_j f_j\|_{L^2}\notag\\
&\leq \|X (\tilde E_j-E_j)  f_j\|_{L^2} + \|Xf_j\|_{L^2} \notag\\
&\lesssim \| u_0 - Q_{\Lambda, \vec c}\|_{L^2} \bigl[ 1 + \|Xf_j\|_{L^2}\bigr] + \|Xf_j\|_{L^2}\notag\\
&\lesssim 1 + \|Xf_j\|_{L^2}.
\end{align}
The utility of these estimates stems from the identities
$$
\tilde c_j = \Re \langle \tilde f_j, X\tilde E_j \tilde f_j\rangle \qtq{and} c_j = \Re \langle  f_j, X E_j  f_j\rangle.
$$
In this way, we may bound
\begin{align*}
|\tilde c_j - c_j|
&\leq \bigl| \langle \tilde f_j, X(\tilde E_j-E_j)\tilde f_j\rangle\bigr| +\bigl| \langle \tilde f_j-f_j, XE_j\tilde f_j\rangle\bigr| +\bigl| \langle f_j, XE_j(\tilde f_j-f_j)\rangle\bigr|\\
&\lesssim  \| u_0 - Q_{\Lambda, \vec c}\|_{L^2} \bigl[1 +\|X\tilde f_j\|_{L^2}\bigr] + \|\tilde f_j-f_j\|_{L^2}  \|Xf_j\|_{L^2}\\
&\lesssim  \| u_0 - Q_{\Lambda, \vec c}\|_{L^2}\bigl[1 +\|Xf_j\|_{L^2}\bigr]\\
&\lesssim  \| u_0 - Q_{\Lambda, \vec c}\|_{L^2},
\end{align*}
where we used \eqref{eigen close}, \eqref{155}, \eqref{1055}, and Corollary~\ref{C: f in X}.
\end{proof}

\begin{prop}[Asymptotics off the real axis]\label{P:Asym_profile}
Let $u(t)$ denote the solution to \eqref{BO} with data $u(0)=u_0\in L^2(\R)$. Given $\la \in \R$, we have:\\
{\upshape (i)} If $\lambda$ is not eigenvalue of $L_{u_0}$, then for any $z \in \C_+$,
\begin{equation}\label{E:Vanishing_Asymptotic_Profile}
    \lim_{|t| \to \infty} \Pi u(t, z - 2t \la) = 0.
\end{equation}
{\upshape (ii)} If $\lambda$ is an eigenvalue of $L_{u_0}$ and $f$ is an $L^2$-normalized eigenfunction associated with $\la$, then for any $z \in \C_+$,
\begin{equation}\label{E:Asymptotic_Profile}
    \lim_{|t| \to \infty} \Pi u(t, z - 2t \la) = \tfrac{2i|\la|}{2|\la|(z-c) + i }, \qtq{where} c: = \Re\langle f, Xf \rangle .
\end{equation}
\end{prop}

\begin{proof}
We recall from \eqref{EFG} that 
\begin{equation*}
    \Pi u(t, z - 2t \la) = \tfrac{1}{2\pi i} I_+ \Bigl( \bigl[X - 2t (L_{u_0} - \la) - z\bigr]^{-1} \Pi u_{0} \Bigr).
\end{equation*}
We will proceed by prudently approximating $u_0$ and using \eqref{BEFG} to control the remainder.  In Case~(ii), however, we must first extract a component of $u_0$ connected to the eigenfunction $f$ corresponding to the eigenvalue $\lambda$.

Recalling that eigenvalues are simple, in Case~(ii) we write
\begin{equation}\label{76}
    \CS u_0 = \langle f,\CS u_0\rangle f + u_0^\perp \qtq{and} f = \tfrac1{\langle f, Xf \rangle - z} \bigl[X - 2t (L_{u_0} - \la) - z\bigr] f  - g,
\end{equation}
where both
\begin{equation}\label{77}
u_0^\perp \qtq{and} g := \tfrac1{\langle f, Xf \rangle - z}[X f - \langle f, X f \rangle f] \qtq{lie in} \Ker( L_{u_0} - \lambda) ^ \perp.
\end{equation}
Combining these relations and setting $v_0:=u_0^\perp - \langle f,\CS u_0\rangle g \in  L^2_+(\R)$, we deduce that $v_0 \in \Ker( L_{u_0} - \lambda) ^ \perp$ and
\begin{equation}
    \bigl[X - 2t (L_{u_0} - \la) - z\bigr]^{-1} \Pi u_{0}
    	= \tfrac{\langle f,\CS u_0\rangle}{\langle f, Xf \rangle - z} f + \bigl[X - 2t (L_{u_0} - \la) - z\bigr]^{-1} v_0 .
\end{equation}
In Case~(i), we define instead $v_0:=\CS u_0 \in \Ker( L_{u_0} - \lambda) ^ \perp$.

Proceeding in this way, we may unify the two cases and reduce our obligations to proving the following: For each $z\in \C_+$ and $v_0 \in \Ker( L_{u_0} - \lambda) ^ \perp$,
\begin{equation}\label{goal78}
    I_+ \Bigl( \bigl[X - 2t (L_{u_0} - \la) - z\bigr]^{-1} v_0 \Bigr) \longrightarrow0 \qtq{as $t\to\pm\infty$}
\end{equation}
and, secondly, for any $L^2$-normalized eigenvector $f\in \Ker( L_{u_0} - \lambda)$,
\begin{equation}\label{goal79}
	\tfrac{1}{2\pi i}I_+ \Bigl( \tfrac{\langle f,\CS u_0\rangle}{\langle f, Xf \rangle - z} f \Bigr) = \tfrac{2i|\la|}{2|\la|(z-c) + i }, \qtq{where} c:=\Re\langle f, X f\rangle.
\end{equation}

We begin with the second obligation \eqref{goal79}.  While $c$ captures the `spatial center' of the eigenfunction, the imaginary part is related to the eigenvalue:
\begin{align}
    \Im\langle f, X f\rangle 
    & = \int_0^\infty \Im \Bigl[\mkern2mu \overline{\widehat{f}(\xi)}\, i\p_\xi \widehat{f}(\xi)\Bigr] \,d\xi = \tfrac{1}{2} \int_0^\infty \p_\xi \big| \widehat{f}(\xi) \big|^2  \,d\xi = -\tfrac{1}{2} |\widehat{f}(0)|^2 
\end{align} 
and so by the Wu relations \eqref{Wu relation Hs},
\begin{equation}
\Im\langle f, X f\rangle  = \tfrac{1}{2\lambda}  .
\end{equation}
The Wu relations also yield
\begin{equation}
\langle f, \CS u_0 \rangle I_+(f) = - \lambda^{-1} |\langle f, \CS u_0 \rangle|^2   = 2\pi .
\end{equation}
The claim \eqref{goal79} now follows from elementary manipulations.

This leaves us to demonstrate \eqref{goal78}.  As $L_{u_0}$ is selfadjoint,
\begin{equation}
\Ker( L_{u_0} - \lambda) ^ \perp = \overline{ \Range( L_{u_0} - \lambda) }.
\end{equation}
Thus, in both cases, $v_0$ may be well approximated arbitrarily well (in $L^2$-sense) by $( L_{u_0} - \lambda) h$ with $h\in H^1_+(\R)$.  Recall that $( L_{u_0} - \lambda)$ defines a bounded map from $H^1_+(\R)$ to $L^2_+(\R)$ and that the collection of $h$ with $\widehat h\in C^\infty_c((0,\infty))$ is dense in $H^1_+(\R)$.  In this way, we see that for any $\eps>0$, we may choose $h_\eps$ and $r_\eps$ so that
\begin{equation}
v_0 = ( L_{u_0} - \lambda) h_\eps + r_\eps \qtq{with} \widehat h_\eps\in C^\infty_c((0,\infty)) \qtq{and} \| r_\eps \|_{L^2_+} < \eps .
\end{equation}

The contribution of $r_\eps$ to \eqref{goal78} is controlled by \eqref{BEFG}.  Concretely,  
\begin{equation}\label{E: I+(r)}
    \left| I_+ \Bigl( \bigl[X - 2t (L_{u_0} - \la) - z\bigr]^{-1} r_\eps \Bigr) \right| \lesssim (\Im z )^{-1/2} \|r_\eps\|_{L^2_+} \lesssim_z \eps,
\end{equation}
uniformly for $t \in \R$.

By construction, $h_\eps\in \D(X)\cap H^1_+(\R)$ and $I_+(h_\eps)=0$; moreover,
\begin{equation}
 v_0-r_\eps = ( L_{u_0} - \lambda) h_\eps = -\tfrac1{2t} \bigl[X - 2t (L_{u_0} - \la) - z\bigr] h_\eps + \tfrac1{2t} \bigl[X - z\bigr] h_\eps.
\end{equation}
From this and \eqref{BEFG}, we deduce that 
\begin{equation}\label{E: I+(v-r)}
    \left| I_+ \Bigl( \bigl[X - 2t (L_{u_0} - \la) - z\bigr]^{-1} (v_0-r_\eps) \Bigr) \right| \lesssim_z \tfrac{1}{|t|} \Bigl[ \|X h_\eps\|_{L^2} + \| h_\eps\|_{L^2}\Bigr].
\end{equation}
As $\eps>0$ was arbitrary, claim \eqref{goal78} follows by combining \eqref{E: I+(r)} and \eqref{E: I+(v-r)}.
\end{proof}

\begin{Rk}
While we have focused on translations by $-2\lambda t$, our rate of convergence would actually allow translations that are $-2\lambda t+ o(t)$.
\end{Rk}

We are finally in a position to complete the proof of Theorem~\ref{T:Intro Strong_loc_cv_L2}, which we restate here for the convenience of the reader:

\begin{theorem}\label{T:Strong_loc_cv_L2}
Let $u(t)$ denote the solution to \eqref{BO} with data $u(0)=u_0\in L^2(\R)$. Given $\la \in \R$, we have:\\
{\upshape (i)} If $\lambda$ is not eigenvalue of $L_{u_0}$, then for any $R>0$,
\begin{equation*}
\lim_{|t| \to \infty}\, \|  u(t, x - 2t \la) \|_{L^2([-R,R])} = 0.
\end{equation*}
{\upshape (ii)} If $\lambda$ is an eigenvalue of $L_{u_0}$ and $f$ is an $L^2$-normalized eigenfunction associated with $\la$, then for any $R>0$,
\begin{equation*}
\lim_{|t| \to \infty} \, \left\| u(t, x - 2t \la) - \tfrac{4|\la|}{4\lambda^2(x-c)^2 + 1 } \right\|_{L^2([-R,R])} = 0, \qtq{where} c: = \Re\langle f, Xf \rangle .
\end{equation*}
\end{theorem}

\begin{proof}[Proof of Theorem~\ref{T:Strong_loc_cv_L2}]
Defining 
\begin{equation*}
    u_{\infty, \la}(x) := 
    \begin{cases} 
      \frac{2i|\la|}{2|\la|(x-c ) + i } & \text{if } \la \in \sigma_d(L_{u_0}), \\
      0 & \text{if } \la \notin \sigma_d(L_{u_0}),
    \end{cases}
\end{equation*}
Proposition \ref{P:Asym_profile} yields
\begin{align}\label{315}
\lim_{|t| \to \infty} \Pi u(t, z - 2t \la) = u_{\infty, \la}(z) \qtq{for each} z\in \C_+.
\end{align}

To establish convergence on the real line in the local $L^2$ topology, we rely on the precompactness of the set
\begin{equation*}
    \mathcal{F}_\la = \bigl\{ \Pi u(t, \cdot - 2t\la)  :\, t \in \R \bigr\}
\end{equation*}
in $L^2([-R,R])$ for any $R>0$. A fundamental property of the Benjamin--Ono flow established in  \cite{Killip2024} is that the orbit $\{u(t) : \, t \in \R\}$ of the solution is $L^2$-equicontinuous. Because spatial translation is an isometry that preserves the frequency distribution, the family $\mathcal{F}_\la$ is also $L^2$-equicontinuous. This $L^2$-equicontinuity ensures that high-frequency projections of functions in $\mathcal F_\la$ are uniformly small. Concretely, for any $\eps>0$, there exists $M=M(\eps)$ so that
\begin{align}\label{118}
\sup_{f\in \mathcal F_\la}\, \|P_{>M} f\|_{L^2}< \eps.
\end{align}

On the other hand, by the Rellich--Kondrachov theorem, the family of functions $\{P_{\leq M} \Pi u(t, \cdot - 2t\la):  t \in \R\}$ is precompact in $L^2([-R, R])$ for any finite $R > 0$. Combining this with \eqref{118}, we deduce that the family $\mathcal{F}_\la$ is precompact in $L^2([-R, R])$.

Consequently, for any sequence of times $t_n \to \infty$, there exists a subsequence (which we still denote by $t_n$) and a limit profile $w_{\infty, \la} \in L^2_{\text{loc}}(\R)$ such that
\begin{equation}\label{119}
    \Pi u(t_n, \cdot - 2t_n\la) \longrightarrow w_{\infty, \la} \quad \text{in } L^2([-R, R]) \text{ for any } R > 0.
\end{equation}

To uniquely identify the limit $w_{\infty, \la}$, we use the extension to the upper half-plane. By the conservation of mass, the sequence $\Pi u(t_n, \cdot - 2t_n\la)$ is uniformly bounded in $L^2(\R)$. As it converges in $L^2_{\text{loc}}(\R)$ to $w_{\infty, \la}$, it must converge weakly to $w_{\infty, \la}$ in $L^2(\R)$. As evaluation at a point in the upper half-plane is a bounded linear functional on $L^2(\R)$, we find that
\begin{equation*}
    \lim_{n \to \infty} \Pi u(t_n, z - 2t_n\la) = w_{\infty, \la}(z) \qtq{for all} z\in \C_+.
\end{equation*}
Comparing this with \eqref{315}, we conclude that the holomorphic extensions of $w_{\infty, \la}$ and $u_{\infty, \la}$ agree and so $w_{\infty, \la}(x) = u_{\infty, \la}(x)$ for almost every $x\in\R$. 

As every sequence $t_n \to \infty$ admits a subsequence that converges strongly in $L^2([-R,R])$ to the same limit $u_{\infty, \la}$, we conclude that 
$$
\lim_{|t| \to \infty} \, \bigl\| \Pi u(t, x - 2t \la) - u_{\infty, \la}(x) \bigr\|_{L^2([-R,R])} = 0.
$$ 
Finally, as the solution to \eqref{BO} is real-valued, we have $u = 2\Re \Pi u$.  The claim now follows from the observation
\begin{align*}
2\Re u_{\infty, \la} (x)  &= \tfrac{4|\la|}{4\lambda^2(x-c)^2 + 1 } = Q_{\la, c}(x) \qtq{whenever} \la\in \sigma_d(L_{u_0}). \qedhere
\end{align*}
\end{proof}

\section{Asymptotic Stability}\label{S:Asymptotic_Stability}

In this section, we establish the asymptotic stability of Benjamin--Ono multisolitons. Concretely, we will show that an $L^2$-perturbation of an $N$-multisoliton leads to an $L^2$-solution $u(t)$ that decomposes asymptotically into the sum of $N$ main traveling solitary waves. By evaluating the flow along the characteristic trajectory of each primary eigenvalue, we can isolate the asymptotic profiles.

\begin{theorem}[Asymptotic Stability]\label{T:Main_Asymptotic}
Fix an integer $N\geq 1$ and a set $\Lambda$ of negative parameters $\la_1 < \dots < \la_N$. There exists $\delta>0$ such that if the initial data $u_0 \in L^2(\R)$ satisfies
\begin{equation}\label{907}
\| u_0 - Q_{\Lambda, \vec c} \|_{L^2} < \delta \qtq{for some} \vec{c}= (c_1, \ldots, c_N)\in \R^N,
\end{equation}
then there exists a set $\tilde \Lambda$ of negative parameters $\tilde{\la}_1 < \dots < \tilde{\la}_N$ and there exist asymptotic spatial centers $\tilde c_1, \ldots, \tilde c_N \in \R$ such that the solution $u(t)$ to \eqref{BO} with initial data $u(0)=u_0$ satisfies
\begin{align}\label{E:Main_Asymptotic_Limit}
\lim_{|t| \to \infty}\, \bigl\| u(t, \cdot - 2t\tilde\la_j) - Q_{\tilde\Lambda,\, \tilde c(t)}(\cdot - 2t\tilde \la_j) \bigr\|_{L^2([-R, R])} = 0
\end{align}
for every $1\leq j \leq N$ and $R > 0$. Here, $\tilde c_j(t) := \tilde c_j - 2t\tilde \lambda_j$, mirroring \eqref{MS3}, and $\tilde c_j$ is given by \eqref{tilde c}. Moreover, we have the stability estimates
\begin{align}\label{E:337}
\limsup_{|t| \to \infty}\, \bigl\| u(t) - Q_{\tilde\Lambda,\, \tilde{c}(t)}\bigr\|_{L^2}^2&\lesssim \delta
\end{align}
and 
\begin{align}\label{344}
\sup_{1\leq j\leq N} \, |\lambda_j - \tilde \lambda_j| + \sup_{1\leq j\leq N} \, |c_j - \tilde c_j| \lesssim \delta.
\end{align}
\end{theorem}

\begin{proof}
Choosing $\delta>0$ small as in Proposition~\ref{P:Spectral_bounds}, we may invoke this result to find $N$ eigenvalues $\tilde \la_1<\ldots <\tilde \la_N$ of $L_{u_0}$ together with $N$ spatial shifts
\begin{equation}\label{tilde c}
\tilde c_j:= \Re\langle \tilde f_j, X\tilde f_j\rangle,
\end{equation}
where $\tilde f_j$ denotes an $L^2$-normalized eigenfunction corresponding to the eigenvalue~$\tilde \la_j$.  Claim \eqref{344} follows from \eqref{E:continuity_lambda} and \eqref{centers close}.

Moreover, for each $1\leq j\leq N$, Theorem~\ref{T:Strong_loc_cv_L2} yields
\begin{equation*}
\lim_{|t| \to \infty} \,\left\| u(t, \cdot - 2t\tilde{\la}_j) - Q_{\tilde{\la}_j, \tilde c_j} \right\|_{L^2([-R, R])} = 0 \qtq{for all} R>0,
\end{equation*}
or equivalently, 
\begin{align}\label{507}
\lim_{|t| \to \infty} \,\left\| u(t, \cdot - 2t\tilde{\la}_j) - Q_{\tilde{\la}_j, \tilde c_j(t)}(\cdot - 2t\tilde{\la}_j) \right\|_{L^2([-R, R])} = 0  \qtq{for all} R>0.
\end{align}

By direct analysis of the multisoliton formula \eqref{MS1}, Matsuno showed that the $N$ multisoliton $Q_{\tilde\Lambda,\, \tilde{c}(t)}$ decouples asymptotically as $|t|\to \infty$ into the sum of the $N$ solitary waves $Q_{\tilde{\la}_j, \tilde c_j(t)} $.  We need this decoupling result in $L^2$-sense, that is,
\begin{align}\label{508}
\lim_{|t| \to \infty} \,\Bigl\|  Q_{\tilde\Lambda,\, \tilde{c}(t)}- \sum_{j=1}^N Q_{\tilde{\la}_j, \tilde c_j(t)} \Bigr\|_{L^2} = 0;
\end{align}
this was established by the authors in \cite[Proposition 3.2]{Killip2024}.  It is essential here that the parameters $\tilde \la_j$ are distinct, which causes the solitary waves $Q_{\tilde{\la}_j, \tilde c_j(t)}$ to asymptotically decouple as $|t|\to \infty$. Claim \eqref{E:Main_Asymptotic_Limit} follows from \eqref{507} and \eqref{508}.

We now turn to \eqref{E:337}. Using the conservation of $L^2$ norm by solutions to \eqref{BO}, we may write
\begin{align*}
\bigl\| u(t) - Q_{\tilde\Lambda,\, \tilde{c}(t)}\bigr\|_{L^2}^2
&= \|u(t)\|_{L^2}^2 + \bigl\|Q_{\tilde\Lambda,\, \tilde{c}(t)}\bigr\|_{L^2}^2 - 2\bigl\langle u(t), Q_{\tilde\Lambda,\, \tilde{c}(t)}\bigr\rangle\\
& = \|u_0\|_{L^2}^2 + \bigl\|Q_{\tilde\Lambda,\, \tilde{c}}\bigr\|_{L^2}^2 - 2\bigl\langle u(t), Q_{\tilde\Lambda,\, \tilde{c}(t)}\bigr\rangle.
\end{align*}
Using \eqref{507} and \eqref{508}, we obtain
\begin{align*}
\lim_{|t|\to \infty} \,\bigl\langle u(t), Q_{\tilde\Lambda,\, \tilde{c}(t)}\bigr\rangle
&= \sum_{j=1}^N\, \lim_{|t|\to \infty} \,\bigl\langle u(t, x-2t\la_j), Q_{\tilde\la_j,\, \tilde c_j}\bigr\rangle\\
&= \sum_{j=1}^N\, \bigl\|Q_{\tilde\la_j,\, \tilde c_j}\bigr\|_{L^2([-R,R])}^2 + o(1) \qtq{as} R\to\infty.
\end{align*}
Letting $R\to \infty$ we deduce
\begin{align*}
\lim_{|t|\to \infty} \,\bigl\| u(t) - Q_{\tilde\Lambda,\, \tilde{c}(t)}\bigr\|_{L^2}^2
&= \|u_0\|_{L^2}^2-\bigl\|Q_{\tilde\Lambda,\, \tilde{c}}\bigr\|_{L^2}^2\\
&= \|u_0\|_{L^2}^2-\bigl\|Q_{\Lambda,\, \vec c}\bigr\|_{L^2}^2 + 2\pi \sum_{j=1}^N |\la_j|-|\tilde \la_j|,
\end{align*}
where we used \eqref{soliton mass} to obtain the last inequality. Claim \eqref{E:337} now follows from \eqref{907} and \eqref{344}.
\end{proof}

\bibliographystyle{amsplain}
\bibliography{BOrefs}

\providecommand{\bysame}{\leavevmode\hbox to3em{\hrulefill}\thinspace}
\providecommand{\MR}{\relax\ifhmode\unskip\space\fi MR }
\providecommand{\MRhref}[2]{%
  \href{http://www.ams.org/mathscinet-getitem?mr=#1}{#2}
}
\providecommand{\href}[2]{#2}
\begin{thebibliography}{10}

\bibitem{Albert1992}
J.~P. Albert, \emph{Positivity properties and stability of solitary-wave
  solutions of model equations for long waves}, Comm. PDE \textbf{17} (1992),
  no.~1-2, 1--22. \MR{1151253}

\bibitem{Albert1999}
\bysame, \emph{Concentration compactness and the stability of solitary-wave
  solutions to nonlocal equations}, Applied analysis ({B}aton {R}ouge, {LA},
  1996), Contemp. Math., vol. 221, Amer. Math. Soc., Providence, RI, 1999,
  pp.~1--29. \MR{1647189}

\bibitem{MR887857}
J.~P. Albert, J.~L. Bona, and D.~B. Henry, \emph{Sufficient conditions for
  stability of solitary-wave solutions of model equations for long waves},
  Phys. D \textbf{24} (1987), no.~1-3, 343--366. \MR{887857}

\bibitem{BKV25}
R.~Badreddine, R.~Killip, and M.~Visan, \emph{Orbital stability of
  {B}enjamin--{O}no multisolitons}, Preprint arXiv:2509.14153 (2025).

\bibitem{Benjamin1967}
T.~B. Benjamin, \emph{Internal waves of permanent form in fluids of great
  depth}, J. Fluid Mech. \textbf{29} (1967), no.~3, 559--592.

\bibitem{MR715035}
D.~P. Bennett, R.~W. Brown, S.~E. Stansfield, J.~D. Stroughair, and J.~L. Bona,
  \emph{The stability of internal solitary waves}, Math. Proc. Cambridge
  Philos. Soc. \textbf{94} (1983), no.~2, 351--379. \MR{715035}

\bibitem{Bock1979}
T.~L. Bock and M.~D. Kruskal, \emph{A two-parameter {M}iura transformation of
  the {B}enjamin--{O}no equation}, Phys. Lett. A \textbf{74} (1979), no.~3--4,
  173--176. \MR{591320}

\bibitem{MR2082818}
J.~L. Bona, Y.~Liu, and N.~V. Nguyen, \emph{Stability of solitary waves in
  higher-order {S}obolev spaces}, Commun. Math. Sci. \textbf{2} (2004), no.~1,
  35--52. \MR{2082818}

\bibitem{MR897729}
J.~L. Bona, P.~E. Souganidis, and W.~A. Strauss, \emph{Stability and
  instability of solitary waves of {K}orteweg--de {V}ries type}, Proc. Roy.
  Soc. London Ser. A \textbf{411} (1987), no.~1841, 395--412. \MR{897729}

\bibitem{Case1979}
K.~M. Case, \emph{Properties of the {B}enjamin--{O}no equation}, J. Math. Phys.
  \textbf{20} (1979), no.~5, 972--977. \MR{531298}

\bibitem{MR516327}
H.~H. Chen, Y.~C. Lee, and N.~R. Pereira, \emph{Algebraic internal wave
  solitons and the integrable {C}alogero-{M}oser-{S}utherland {$N$}-body
  problem}, Phys. Fluids \textbf{22} (1979), no.~1, 187--188. \MR{516327}

\bibitem{chen2025explicit}
X.~Chen, \emph{Explicit formula for the {B}enjamin--{O}no equation with square
  integrable and real valued initial data and applications to the zero
  dispersion limit}, Pure and Applied Analysis \textbf{7} (2025), no.~1,
  101--126.

\bibitem{Davis1967}
R.~E. Davis and A.~Acrivos, \emph{Solitary internal waves in deep water}, J.
  Fluid Mech. \textbf{29} (1967), no.~3, 593--607.

\bibitem{GG26}
L.~Gassot and P.~G{\'e}rard, \emph{Infinite-order multisoliton solutions to the
  {B}enjamin--{O}no equation and soliton resolution}, Preprint arXiv:2603.15419
  (2026).

\bibitem{GGM26}
L.~Gassot, P.~G{\'e}rard, and P.~D. Miller, \emph{A proof of the soliton
  resolution conjecture for the {B}enjamin--{O}no equation}, Preprint
  arXiv:2601.10488 (2026).

\bibitem{MR4662323}
P.~G\'erard, \emph{An explicit formula for the {B}enjamin--{O}no equation},
  Tunis. J. Math. \textbf{5} (2023), no.~3, 593--603. \MR{4662323}

\bibitem{GTT2009}
S.~Gustafson, H.~Takaoka, and T.-P. Tsai, \emph{Stability in {$H^{1/2}$} of the
  sum of {$K$} solitons for the {B}enjamin--{O}no equation}, J. Math. Phys.
  \textbf{50} (2009), no.~1, 013101, 14. \MR{2492606}

\bibitem{Joseph1977}
R.~I. Joseph, \emph{Multi-soliton-like solutions to the {B}enjamin--{O}no
  equation}, J. Mathematical Phys. \textbf{18} (1977), no.~12, 2251--2258.
  \MR{452116}

\bibitem{KenMart2009}
C.~E. Kenig and Y.~Martel, \emph{Asymptotic stability of solitons for the
  {B}enjamin--{O}no equation}, Rev. Mat. Iberoam. \textbf{25} (2009), no.~3,
  909--970. \MR{2590690}

\bibitem{Killip2024}
R.~Killip, T.~Laurens, and M.~Vi\c{s}an, \emph{Sharp well-posedness for the
  {B}enjamin--{O}no equation}, Invent. Math. \textbf{236} (2024), no.~3,
  999--1054. \MR{4743514}

\bibitem{KLV2025}
\bysame, \emph{Scaling-critical well-posedness for continuum {C}alogero-{M}oser
  models on the line}, Commun. Am. Math. Soc. \textbf{5} (2025), 284--320.
  \MR{4922705}

\bibitem{LanWang2025}
Y.~Lan and Z.~Wang, \emph{Stability of multi-solitons for the {Benjamin--Ono}
  equation}, Preprint arXiv:2302.14205.

\bibitem{Lax1968}
P.~D. Lax, \emph{Integrals of nonlinear equations of evolution and solitary
  waves}, Comm. Pure Appl. Math. \textbf{21} (1968), 467--490. \MR{235310}

\bibitem{MR1220540}
J.~H. Maddocks and R.~L. Sachs, \emph{On the stability of {K}d{V}
  multi-solitons}, Comm. Pure Appl. Math. \textbf{46} (1993), no.~6, 867--901.
  \MR{1220540}

\bibitem{Matsuno1979}
Y.~Matsuno, \emph{Exact multi-soliton solution of the {Benjamin--Ono}
  equation}, J. Phys. A \textbf{12} (1979), no.~4, 619--621.

\bibitem{MR759718}
\bysame, \emph{Bilinear transformation method}, Mathematics in Science and
  Engineering, vol. 174, Academic Press, Inc., Orlando, FL, 1984. \MR{759718}

\bibitem{Matsuno2006}
\bysame, \emph{The {L}yapunov stability of the {$N$}-soliton solutions in the
  {L}ax hierarchy of the {B}enjamin--{O}no equation}, J. Math. Phys.
  \textbf{47} (2006), no.~10, 103505, 13pp. \MR{2268871}

\bibitem{MR1442235}
Y.~Matsuno and D.~J. Kaup, \emph{Linear stability of multiple internal solitary
  waves in fluids of great depth}, Phys. Lett. A \textbf{228} (1997), no.~3,
  176--181. \MR{1442235}

\bibitem{MeissPereira}
J.~D. Meiss and N.~R. Pereira, \emph{Internal wave solitons}, Phys. Fluids
  \textbf{21} (1978), no.~4, 700--702.

\bibitem{Nakamura1979}
A.~Nakamura, \emph{B\"{a}cklund transform and conservation laws of the
  {B}enjamin--{O}no equation}, J. Phys. Soc. Japan \textbf{47} (1979), no.~4,
  1335--1340. \MR{550203}

\bibitem{MR2202311}
A.~Neves and O.~Lopes, \emph{Orbital stability of double solitons for the
  {B}enjamin--{O}no equation}, Comm. Math. Phys. \textbf{262} (2006), no.~3,
  757--791. \MR{2202311}

\bibitem{Ono1975}
H.~Ono, \emph{Algebraic solitary waves in stratified fluids}, J. Phys. Soc.
  Japan \textbf{39} (1975), no.~4, 1082--1091. \MR{398275}

\bibitem{Sun2021}
R.~Sun, \emph{Complete integrability of the {B}enjamin--{O}no equation on the
  multi-soliton manifolds}, Comm. Math. Phys. \textbf{383} (2021), no.~2,
  1051--1092. \MR{4239837}

\bibitem{MR886343}
M.~I. Weinstein, \emph{Existence and dynamic stability of solitary wave
  solutions of equations arising in long wave propagation}, Comm. Partial
  Differential Equations \textbf{12} (1987), no.~10, 1133--1173. \MR{886343}

\bibitem{MR3484397}
Y.~Wu, \emph{Simplicity and finiteness of discrete spectrum of the
  {B}enjamin--{O}no scattering operator}, SIAM J. Math. Anal. \textbf{48}
  (2016), no.~2, 1348--1367. \MR{3484397}

\end{thebibliography}

\end{document}